\documentclass[reqno,10pt,centertags]{amsart} 
\usepackage{amsmath,amsthm,amscd,amssymb,latexsym,esint,upref,stmaryrd,
enumerate,verbatim,yfonts,bm,mathtools}
\usepackage{color}
\usepackage{hyperref} 
\newcommand*{\mailto}[1]{\href{mailto:#1}{\nolinkurl{#1}}}

\newcommand{\bbC}{{\mathbb{C}}}

\newcommand{\bbN}{{\mathbb{N}}}

\newcommand{\bbR}{{\mathbb{R}}}
\newcommand{\bbS}{{\mathbb{S}}}

\newcommand{\cB}{{\mathcal B}}

\newcommand{\cH}{{\mathcal H}}

\newcommand{\cR}{{\mathcal R}}

\newcommand{\beq}{\begin{equation}}
\newcommand{\enq}{\end{equation}}

\newcommand{\z}{\zeta}

\DeclareMathOperator{\dom}{dom}

\DeclareMathOperator*{\sgn}{sgn}

\renewcommand{\Re}{\text{\rm Re}}
\renewcommand{\Im}{\text{\rm Im}}
\renewcommand{\ln}{\text{\rm ln}}

\newcommand{\loc}{\operatorname{loc}}

\newcommand{\dott}{\,\cdot\,}

\newcommand{\no}{\notag}
\newcommand{\lb}{\label}
\newcommand{\f}{\frac}

\newcommand{\ol}{\overline}

\newcommand{\wti}{\widetilde}
\newcommand{\Oh}{O}

\newcommand{\bi}{\bibitem}

\renewcommand{\ge}{\geqslant}

\newcommand{\WtwoN}{[W^{2,2}(\bbR^n)]^N} 

\let\geq\geqslant
\let\leq\leqslant

\newcommand{\LN}{[L^2(\bbR^n)]^N} 
\newcommand{\WoneN}{[W^{1,2}(\bbR^n)]^N} 

\makeatletter
\def\theequation{\@arabic\c@equation}

\allowdisplaybreaks 
\numberwithin{equation}{section}

\newtheorem{theorem}{Theorem}[section]

\newtheorem{lemma}[theorem]{Lemma}

\newtheorem{definition}[theorem]{Definition}
\newtheorem{hypothesis}[theorem]{Hypothesis}

\theoremstyle{remark}
\newtheorem{remark}[theorem]{Remark}

\begin{document}

\title[Absence of Threshold Resonances]{On Absence of Threshold Resonances for Schr\"odinger and Dirac Operators} 

\author[F.\ Gesztesy]{Fritz Gesztesy}
\address{Department of Mathematics, 
Baylor University, One Bear Place \#97328,
Waco, TX 76798-7328, USA}
\email{\mailto{Fritz\_Gesztesy@baylor.edu}}
\urladdr{\url{http://www.baylor.edu/math/index.php?id=935340}}

\author[R.\ Nichols]{Roger Nichols}
\address{Department of Mathematics, The University of Tennessee at Chattanooga, 
415 EMCS Building, Dept. 6956, 615 McCallie Ave, Chattanooga, TN 37403, USA}
\email{\mailto{Roger-Nichols@utc.edu}}
\urladdr{\url{http://www.utc.edu/faculty/roger-nichols/index.php}}

\dedicatory{Dedicated with great pleasure to Gis\`ele Goldstein}

\date{\today}
\thanks{Appeared in {\it Discrete Cont. Dyn. Syst, Ser. S}, {\bf 13}, 3427--3460 (2020).}
\subjclass[2010]{Primary: 35J10, 35Q41, 45P05, 47A11, 47G10; Secondary: 35Q40, 47A10, 81Q10.}
\keywords{Threshold resonances, threshold eigenvalues, Schr\"odinger operators, Dirac operators.}

\begin{abstract} 
Using a unified approach employing a homogeneous Lippmann-Schwinger-type equation satisfied by resonance functions and basic facts on Riesz potentials, we discuss the absence of threshold resonances for Dirac and Schr\"odinger operators with sufficiently short-range interactions in general space dimensions.  

More specifically, assuming a sufficient power law decay of potentials, we derive the absence of zero-energy resonances for massless Dirac operators in space dimensions $n \geq 3$, the absence of resonances at $\pm m$ for massive Dirac operators (with mass $m > 0$) in dimensions $n \geq 5$, and recall the well-known case of absence of zero-energy resonances for Schr\"odinger operators in dimension $n \geq 5$. 
\end{abstract}

\maketitle



\section{Introduction} \lb{s1} 

{\it Happy Birthday, Gis\`ele, we sincerely hope that our modest contribution to threshold resonances of Schr\"odinger and Dirac operators will create some joy.} 

The principal purpose of this paper is a systematic investigation of threshold resonances, more precisely, the absence of the latter, for massless and massive Dirac operators, and for Schr\"odinger operators in general space dimensions $n \in \bbN$, $n \geq 2$, given sufficiently fast decreasing interaction potentials at infinity (i.e., short-range interactions). 

In the case that $(- \infty, E_0)$ and/or $(E_1,E_2)$, $E_j \in \bbR$, $j = 0,1,2$,  are essential (resp., absolutely continuous) spectral gaps of a self-adjoint unbounded operator $A$ in some complex, separable Hilbert space 
$\cH$, then typically, the numbers $E_0$ and/or the numbers $E_1, E_2$ are called threshold energies of $A$. The intuition behind the term ``threshold'' being that if the coupling constant of an appropriate interaction potential is varied, then eigenvalues will eventually ``emerge'' out of the continuous spectrum of the operator $A$ and enter the spectral gap $(- \infty, E_0)$ and/or $(E_0, E_1)$. The precise phenomenon behind this intuition is somewhat involved, depending, in particular, on the (possibly singular) behavior of the resolvent of $A$ at the points 
$E_j$, $j=0,1,2$. In the concrete case of Dirac and Schr\"odinger operators with sufficient short-range interactions at hand, this behavior is well understood (see, e.g., \cite{Kl82}, \cite{Kl85}, \cite{KS80} and the literature cited therein) and $E_0 = 0$ is then a threshold for Schr\"odinger operators $h$, with $\sigma_{ess}(h) = [0,\infty)$ (cf.\ \eqref{2.14A}--\eqref{2.18A}), whereas $\pm m$ are thresholds for massive Dirac operators $H(m)$ corresponding to mass $m > 0$, with 
$\sigma_{ess}(H(m)) = (-\infty, -m] \cup [m,\infty)$ (cf.\ \eqref{4.8A}--\eqref{4.11A}). The case of massless Dirac operators $H$ is more intricate as $\sigma_{ess} (H) = \bbR$ (cf.\ \eqref{4.2}, \eqref{4.6A}), and hence $H$ exhibits no spectral gap. However, since the potential coefficients tend to zero at infinity, $0$ is still a threshold point of $H$ that is known to possibly support an eigenvalue and/or a resonance. 

Threshold resonances for Dirac and Schr\"odinger operators then are associated with distributional solutions 
$\psi$ of $h \psi = 0$, or $\Psi_{\pm m}$ of $H(m) \Psi_{\pm m} = \pm m \Psi_{\pm m}$, respectively, 
distributional solutions $\Psi$ of $H \Psi = 0$, that do not belong to the domains of $h, H(m)$, respectively $H$, but these functions are ``close'' to belonging to the respective operator domains in the sense that they (actually, their components) belong to $L^p(\bbR^n)$ for $p \in (p_0, \infty) \cup \{\infty\}$ for appropriate $p_0 > 2$. More precisely, we then recall in our principal Section \ref{s3} that Schr\"odinger operators have no zero-energy resonances in dimensions $n \geq 5$ (a well-known result, cf.\ Remark \ref{r9.8a}\,$(iii)$), that massless Dirac operators have no zero-energy resonances in dimension $n \geq 3$ (the case $n=3$ was well-known, cf.\ 
Remark \ref{r9.8}\,$(iii)$), and that massive Dirac operators have no resonances at $\pm m$ in dimensions $n \geq 5$. We prove these results in a unified manner employing a homogeneous Lippmann--Schwinger-type equation satisfied by threshold resonance functions and the use of basic properties of Riesz potentials. While our power law decay assumptions of the potentials at infinity are not optimal, optimality was not the main motivation for writing this paper. Instead, we wanted to present a simple, yet unified, approach to the question of absence of threshold resonances with the added bonus that our results in the massless case for $n \geq 4$ and in the massive case appear to be new. Since we explicitly permit matrix-valued potentials in the (massless and massive) Dirac case, we note that electric and magnetic potentials are included in our treatment (cf.\ Remark \ref{r3.2}). 

We also emphasize that threshold resonances and/or threshold eigenvalues of course profoundly influence the (singularity) behavior of resolvents at threshold energies, the threshold behavior of (fixed energy, or on-shell) scattering matrices, as well as dispersive (i.e., large time) estimates for wave functions. These issues have received an abundance of attention since the late 1970's in connection with Schr\"odinger operators, and very recently especially in the context of (massless and massive) Dirac operators (the case of massless Dirac operators having applications to graphene). While a complete bibliography in this connection is clearly beyond the scope of our paper, we refer to \cite{AMN99}--\cite{BE03}, \cite[Ch.~4]{BE11}, \cite{BES08}, 
\cite{BV15}--\cite{BGK87}, \cite{El00}--\cite{ES01}, \cite{FLL86}, \cite{GH87}, \cite{Je80}--\cite{KOY15}, \cite{Kl90}--\cite{Mu82}, \cite{Pe08}--\cite{SU15}, \cite{To17}--\cite{ZG13}, 
and the literature therein to demonstrate some of the interest generated by this circle of ideas. We also mention 
the relevance of threshold states in connection with the Witten index for non-Fredholm operators, see, for instance, \cite{CGGLPSZ16b}--\cite{CGPST17}. 

In Section \ref{s2} we present the necessary background for Dirac and Schr\"odinger operators, including a description of Green's functions (resp., matrices) and their threshold asymptotics in the free case, that is, in the absence of interaction potentials in terms of appropriate Hankel functions. Appendix \ref{sA} discusses the failure of our technique to produce essential boundedness (of components) of threshold resonances and eigenfunctions for space dimensions sufficiently large. 

We conclude this introduction with some comments on the notation employed in this paper: 
If $\cH$ denotes a separable, complex Hilbert space, $I_{\cH}$ represents the identity operator in $\cH$, but we simplify this to $I_N$ in the finite-dimensional case $\bbC^N$, $N \in \bbN$.  

If $T$ is a linear operator mapping (a subspace of) a Hilbert space into another, then 
$\dom(T)$ and $\ker(T)$ denote the domain and kernel (i.e., null space) of $T$. 
The spectrum, point spectrum (the set of eigenvalues), and the essential spectrum of a closed linear operator in $\cH$ will be denoted by $\sigma(\, \cdot \,)$, $\sigma_p(\, \cdot \,)$, and $\sigma_{ess}(\, \cdot \,)$, respectively. Similarly, the absolutely continuous and singularly continuous spectrum of a self-adjoint operator in $\cH$ are denoted by $\sigma_{ac}(\, \cdot \,)$ and $\sigma_{sc}(\, \cdot \,)$.

The Banach space of bounded  linear operators on a separable complex Hilbert space $\cH$ is denoted 
by $\cB(\cH)$. 

If $p \in [1,\infty)\cup \{\infty\}$, then $p' \in [1,\infty) \cup \{\infty\}$ denotes its conjugate index, that is, $p':= (1-1/p)^{-1}$. If Lebesgue measure is understood, we simply write $L^p(M)$, $M \subseteq \bbR^n$ measurable, $n \in \bbN$, instead of the more elaborate notation $L^p(M; d^nx)$. 
$\lfloor \, \cdot \, \rfloor$ denotes the floor function on $\bbR$, that is, $\lfloor x \rfloor$ characterizes the largest integer less than or equal to $x \in \bbR$. Finally, if $x=(x_1,\ldots,x_n)\in \bbR^n$, $n \in \bbN$, then we abbreviate $\langle x \rangle := (1+|x|^2)^{1/2}$.

\section{Some Background Material} \lb{s2}

This preparatory section is primarily devoted to various results on Green's functions for the free Schr\"odinger operator (i.e., minus the Laplacian) and free (massive and massless) Dirac operators (i.e., in the absence of interaction potentials) in dimensions $n \in \bbN$, $n \geq 2$.  

To rigorously define the case of free massless $n$-dimensional Dirac operators, $n \geq 2$, we now 
introduce the following set of basic hypotheses assumed for the remainder of this manuscript.

\begin{hypothesis} \lb{h3.1}  Let $ n \in \bbN$, $n\geq 2$. \\[1mm] 
$(i)$ Set $N=2^{\lfloor(n+1)/2\rfloor}$ and let $\alpha_j$, $1\leq j\leq n$, 
$\alpha_{n+1} := \beta$, denote $n+1$ anti-commuting Hermitian $N\times N$ matrices with squares equal to $I_N$, that is, 
\begin{equation}\lb{2.1}
\alpha_j^*=\alpha_j,\quad \alpha_j\alpha_k + \alpha_k\alpha_j = 2\delta_{j,k}I_N, 
\quad 1\leq j,k\leq n+1.
\end{equation}
$(ii)$  Introduce 
in $\LN$ the free massless Dirac operator
\begin{equation}
H_0 = \alpha  \cdot (-i \nabla) = \sum_{j=1}^n \alpha_j (-i \partial_j),\quad \dom(H_0) = \WoneN,  \lb{2.2}
\end{equation}
where $\partial_j = \partial / \partial x_j$, $1 \leq j \leq n$. \\[1mm] 
$(iii)$ Next, consider the self-adjoint matrix-valued potential 
$V = \{V_{\ell,m}\}_{1 \leq \ell,m \leq N}$ satisfying for some fixed $\rho \in (1,\infty)$, $C \in (0,\infty)$, 
\begin{equation}
V \in [L^{\infty} (\bbR^n)]^{N \times N}, \quad 
|V_{\ell,m}(x)| \leq C \langle x \rangle^{- \rho} \, \text{ for a.e.~$x \in \bbR^n$, $1 \leq \ell,m \leq N$.}    \lb{4.1}
\end{equation}
Under these assumptions on $V$, the massless Dirac operator $H$ in $\LN$ is defined via 
\begin{equation}
H = H_0 + V, \quad \dom(H) = \dom(H_0) = \WoneN.  \lb{4.2}
\end{equation}
\end{hypothesis}

We recall that 
\begin{equation}
[L^2(\bbR^n)]^N = L^2(\bbR^n; \bbC^N), \quad [W^{1,2}(\bbR^n)]^N = W^{1,2}(\bbR^n; \bbC^N), \, 
\text{ etc.}
\end{equation}

Then $H_0$ and $H$ are self-adjoint in $\LN$, 
with essential spectrum covering the 
entire real line,
\begin{equation}
\sigma_{ess} (H) = \sigma_{ess} (H_0) = \sigma (H_0) = \bbR,    \lb{4.6A}
\end{equation}
a consequence of relative compactness of $V$ with respect to $H_0$. In addition,
\begin{equation}
\sigma_{ac}(H_0) = \bbR, \quad \sigma_p(H_0) = \sigma_{sc}(H_0) = \emptyset.    \lb{4.7A}
\end{equation}

We also recall that the massive free Dirac operator in $\LN$ associated 
with the mass parameter $m > 0$ is of the form (with $\beta = \alpha_{n+1}$)
\begin{equation}
H_0(m) = H_0 + m \beta, \quad \dom(H_0(m)) = \WoneN, \; m > 0,   \lb{4.8} 
\end{equation}
and the corresponding interacting massive Dirac operator in $\LN$ is given by
\begin{equation}
H(m) = H_0(m) + V = H_0 + m \beta + V, \quad \dom(H(m)) = \WoneN, \; m > 0.    \lb{4.8A} 
\end{equation}
In this case,
\begin{equation}
\sigma_{ess} (H(m)) = \sigma_{ess} (H_0(m)) = \sigma (H_0(m)) = (-\infty,-m] \cup [m,\infty), \quad m > 0, 
\lb{4.10A} 
\end{equation}
and 
\begin{equation}
\sigma_{ac}(H_0(m)) = (- \infty,-m] \cup [m,\infty), \quad 
\sigma_p(H_0(m)) = \sigma_{sc}(H_0(m)) = \emptyset, \quad m > 0.     \lb{4.11A}
\end{equation}

In the special one-dimensional case $n=1$, one can choose for $\alpha_1$ either a real constant or one of the three Pauli matrices. Similarly, in the massive case, $\beta$ would typically be a second Pauli matrix (different from $\alpha_1$). For simplicity we confine ourselves to $n \in \bbN$, 
$n \geq 2$, in the following. 

Employing the relations \eqref{2.1}, one observes that 
\begin{equation} 
H_0^2 = - I_N \Delta, \quad \dom\big(H_0^2\big) = \WtwoN.   \lb{2.7} 
\end{equation}

\begin{remark} \lb{r3.2}
Since we permit a (sufficiently decaying) matrix-valued potential $V$ in $H$, this includes, in particular, the case of electromagnetic interactions introduced via minimal coupling, that is, $V$ 
describes also special cases of the form,
\begin{equation}
H(v,A) := \alpha \cdot (-i \nabla - A) + v I_N = H_0 + [v I_N - \alpha \cdot A], \quad 
\dom(H(v,A)) = \WoneN,
\end{equation}
where $(v,A)$ represent the electromagnetic potentials on $\bbR^n$, with 
$v: \bbR^n \to \bbR$, $v \in L^{\infty}(\bbR^n)$, $A = (A_1,\dots,A_n)$, $A_j: \bbR^n \to \bbR$, 
$A_j \in L^{\infty}(\bbR^n)$, $ 1 \leq j \leq n$, and for some fixed $\rho > 1$, $C \in (0,\infty)$, 
\begin{equation}
|v(x)| + |A_j(x)| \leq C \langle x \rangle^{- \rho}, \quad x \in \bbR^n, \; 1 \leq j \leq n.  
 \lb{4.14}
\end{equation}
The analogous remark applies of course to the massive case with $H_0$ replaced by $H_0(m)$, $m > 0$. 
 \hfill $\diamond$
\end{remark}

We continue with the self-adjoint Laplacian in $L^2(\bbR^n)$,
\begin{equation} 
h_0 = - \Delta, \quad \dom(h_0) = W^{2,2}(\bbR^n),      \lb{2.14} 
\end{equation}
and the interacting Schr\"odinger operator in $L^2(\bbR^n)$ given by
\begin{equation} 
h = h_0 + v, \quad \dom(h) = W^{2,2}(\bbR^n),      \lb{2.14A}   
\end{equation}
where $v \in L^{\infty}(\bbR^n)$, $v$ real-valued, satisfies the estimate \eqref{4.14} (for simplicity). 
In this case,
\begin{equation}
\sigma_{ess} (h) = \sigma_{ess} (h_0) = \sigma (h_0) = [0,\infty),    \lb{2.17A} 
\end{equation}
and 
\begin{equation}
\sigma_{ac}(h_0) = [0,\infty), \quad 
\sigma_p(h_0) = \sigma_{sc}(h_0) = \emptyset.     \lb{2.18A}
\end{equation}

The Green's function of $h_0$, denoted by $g_0(z; \, \cdot \,, \, \cdot \,)$, is of the form,
\begin{align}
& g_{0}(z;x,y) := (h_0 - z I)^{-1}(x,y)   \no \\ 
& \quad = \begin{cases} (i/2) z^{- 1/2} e^{i z^{1/2} |x-y|}, & n =1, \; z\in\bbC\backslash\{0\}, \\[2mm]
(i/4) \big(2\pi z^{-1/2} |x - y|\big)^{(2-n)/2} 
H^{(1)}_{(n-2)/2}\big(z^{1/2}|x - y|\big), & n\ge 2, \; z\in\bbC\backslash\{0\}, 
\end{cases}    \lb{5.2}  \no \\
& \hspace*{6.3cm}   \Im\big(z^{1/2}\big) > 0, 
\; x, y \in\bbR^n, \; x \neq y, 
\end{align}
and for $z=0$, $n \geq 3$, 
\begin{align}
\begin{split} 
g_0(0;x,y) &= \displaystyle{\f{1}{(n-2) \omega_{n-1}} |x - y|^{2-n}}    \\
&= 4^{-1} \pi^{- n/2} \Gamma((n-2)/2) |x - y|^{2-n}, \quad n \geq 3, \;  \; x, y \in\bbR^n, \; x \neq y.   \lb{5.3} 
\end{split} 
\end{align}
Here $H^{(1)}_{\nu}(\, \cdot \,)$ denotes the Hankel function of the first kind 
with index $\nu\geq 0$ (cf.\ \cite[Sect.\ 9.1]{AS72}) and 
$\omega_{n-1}=2\pi^{n/2}/\Gamma(n/2)$ ($\Gamma(\, \cdot \,)$ the Gamma function, 
cf.\ \cite[Sect.\ 6.1]{AS72}) represents the area of the unit sphere $S^{n-1}$ in $\bbR^n$. 

As $z\to 0$, $g_0(z; \, \cdot \,, \, \cdot \,)$ is continuous on the off-diagonal for $n\geq 3$, 
\begin{align}
\begin{split}
\lim_{\substack{z \to 0\\z \in \bbC \backslash \{0\}}}  g_0(z;x,y) = g_0(0;x,y) 
= \f{1}{(n-2) \omega_{n-1}} |x - y|^{2-n},& \\ 
 x, y \in\bbR^n, \, x \neq y, \; n \in \bbN, \; n \ge 3,&    \lb{5.3a} 
 \end{split} 
\end{align} 
but blows up for $n=1$ as 
\begin{align}
& g_0(z;x,y) \underset{\substack{z\to 0 \\ z\in \bbC \backslash\{0\}}}{=} 
(i/2) z^{-1/2} - 2^{-1} |x-y| + \Oh\big(z^{1/2} |x-y|^2\big), \quad x, y \in \bbR, 
\end{align}
and for $n=2$ as
\begin{align} 
\begin{split} 
& g_0(z;x,y) \underset{\substack{z\to 0 \\ z\in \bbC \backslash\{0\}}}{=} 
-\f{1}{2\pi} \ln\big(z^{1/2}|x - y|/2\big)
\big[1+\Oh\big(z|x - y|^2\big)\big] 
+ \f{1}{2\pi} \psi(1)     \\
& \hspace*{2.2cm}  + \Oh\big(|z||x - y|^2\big),   \quad 
 x, y \in\bbR^2, \, x \neq y.    \lb{5.4} 
 \end{split} 
\end{align} 
Here $\psi(w)=\Gamma'(w)/\Gamma(w)$ denotes the digamma function 
(cf. \cite[Sect.\,6.3]{AS72}). 

In connection with the free massive Dirac operator 
$H_0(m) = H_0 + m \, \beta$, $m > 0$, one computes,
\begin{align}
\begin{split} 
(H_0(m) - z I)^{-1} &= (H_0(m) + z I) \big(H_0(m)^2 - z^2 I\big)^{-1}     \\ 
& = (- i \alpha \cdot \nabla + m \, \beta + z I)\big(h_0 - (z^2 - m^2) I\big)^{-1} I_N,    \lb{5.5a} 
\end{split} 
\end{align}
employing 
\begin{equation}
H_0(m)^2 = (h_0 + m^2 I) I_N. 
\end{equation}

Assuming 
\begin{equation}
m > 0, \; z \in \bbC \backslash (\bbR \backslash [-m,m]), \; \Im\big(z^2 - m^2\big)^{1/2} > 0, \; 
x, y \in \bbR^n, \, x \neq y, \; n \in \bbN, \, n \geq 2,    \lb{5.7}
\end{equation}
and exploiting \eqref{5.5a}, one thus obtains for the Green's function 
$G_0(m,z;\, \cdot \,,\, \cdot \,)$ 
of $H_0(m)$, 
\begin{align}
& G_0(m,z;x,y) := (H_0(m) - z I)^{-1}(x,y)    \no \\
& \quad = i 4^{-1} (2 \pi)^{(2-n)/2} |x - y|^{2-n}(m \, \beta + z I_N)  \no \\
& \qquad \times  \big[\big(z^2 - m^2\big)^{1/2} |x - y|\big]^{(n-2)/2} 
H_{(n-2)/2}^{(1)} \big(\big(z^2 - m^2\big)^{1/2} |x - y|\big)     \no \\
& \qquad - 4^{-1} (2 \pi)^{(2-n)/2} |x - y|^{1 - n} \, \alpha \cdot \f{(x - y)}{|x - y|}    \no \\
& \qquad \quad \times \big[\big(z^2 - m^2\big)^{1/2} |x - y|\big]^{n/2} 
H_{n/2}^{(1)} \big(\big(z^2 - m^2\big)^{1/2} |x - y|\big).    
\end{align}
Here we employed the identity (\cite[p.~361]{AS72}),
\begin{equation}
\big[H_{\nu}^{(1)} (\zeta)\big]' = - H_{\nu + 1}^{(1)}(\zeta) + \nu \, \zeta^{-1} H_{\nu}^{(1)}(\zeta), \quad 
\nu, \zeta \in \bbC.  
\end{equation}

We also recall the asymptotic behavior (cf.\ \cite[p.~360]{AS72}, \cite[p.~723--724]{JN01})
\begin{align}
& H_0^{(1)} (\zeta)  \underset{\substack{\zeta\to 0 \\ \zeta\in \bbC \backslash\{0\}}}{=} 
(2i/\pi) \ln(\zeta) + \Oh\big(|\ln(\zeta)| |\zeta|^2\big),  \lb{5.9} \\
& H_{\nu}^{(1)} (\zeta)  \underset{\substack{\zeta\to 0 \\ \zeta\in \bbC \backslash\{0\}}}{=} 
- (i/\pi) 2^{\nu} \Gamma(\nu) \zeta^{- \nu} 
+ \begin{cases} \Oh\big(|\zeta|^{\min(\nu, - \nu + 2)}\big), & \nu \notin \bbN, \\
\Oh\big(|\ln(\zeta)| |\zeta|^{\nu}\big) + \Oh\big(\zeta^{ - \nu + 2}\big), & \nu \in \bbN, \end{cases}  \lb{5.10} \\
& \hspace*{9.6cm} \Re(\nu) > 0,  \no \\
& H_{\nu}^{(1)} (\zeta) \underset{\zeta \to \infty}{=} (2/\pi)^{1/2} \zeta^{-1/2} e^{i [\zeta - (\nu \pi/2) - (\pi/4)]}, 
\quad \nu \geq 0, \; \Im(\zeta) \geq 0.   \lb{5.11} 
\end{align}

Equations \eqref{5.9}, \eqref{5.10} reveal the facts (still assuming \eqref{5.7}),
\begin{align}
& \lim_{\substack{z \to \pm m \\ z \in \bbC \backslash \{\pm m\}}} 
G_0(m,z; x,y) = 4^{-1} \pi^{- n/2} \Gamma((n-2)/2) |x - y|^{2-n} (m \, \beta \pm m I_N)     \lb{5.14} \\
& \quad + i 2^{-1} \pi^{-n/2} \Gamma(n/2) \, \alpha \cdot \f{(x - y)}{|x - y|^n}, 
\quad m > 0, \; x, y \in \bbR^n, \, x \neq y, \; n \in \bbN, \, n \geq 3,    \no \\
& G_0(m,z; x,y) \underset{\substack{z \to \pm m \\ z \in \bbC \backslash \{\pm m\}}}{=}  
- (4 \pi)^{-1} \ln\big(z^2 - m^2\big) (m \, \beta \pm m I_2)        \no \\
&  \quad  - (2 \pi)^{-1} \ln(|x - y|) (m \, \beta \pm m I_2) 
+ i (2 \pi)^{-1} \, \alpha \cdot \f{(x - y)}{|x - y|^2}    \lb{5.15} \\
&  \quad  
+ \Oh\big(\big(z^2 - m^2\big) \ln\big(z^2 - m^2\big)\big), 
\quad m > 0, \; x, y \in \bbR^2, \, x \neq y.  \no
\end{align}
(Here the remainder term $\Oh\big(\big(z^2 - m^2\big) \ln\big(z^2 - m^2\big)\big)$ depends 
on $x, y \in \bbR^2$, but this is of no concern at this point.) 
In particular, $G_0(m,z; \, \cdot \, , \, \cdot \,)$ blows up logarithmically as $z \to \pm m$ 
in dimensions $n=2$, just as $g_0(z, \, \cdot \, , \, \cdot \,)$ does as $z \to 0$. 

By contrast, the massless case is quite different and assuming 
\begin{equation}
z \in \bbC_+, \; x, y \in \bbR^n, \, x \neq y, \; n \in \bbN, \, n \geq 2,    \lb{5.16}
\end{equation}
one computes in the case $m=0$ for the Green's function $G_0(z;\, \cdot \,,\, \cdot \,)$ of $H_0$, 
\begin{align}
& G_0(z;x,y) := (H_0 - z I)^{-1}(x,y)     \no \\ 
& \quad = i 4^{-1} (2 \pi)^{(2-n)/2} |x - y|^{2-n} z \, [z |x - y|]^{(n-2)/2} 
H_{(n-2)/2}^{(1)} (z |x - y|) I_N   \lb{5.17} \\
& \qquad - 4^{-1} (2 \pi)^{(2-n)/2} |x - y|^{1-n} [z |x - y|]^{n/2} H_{n/2}^{(1)} (z |x - y|) \, 
\alpha \cdot \f{(x - y)}{|x - y|}.    \no
\end{align}
The Green's function $G_0(z;\, \cdot \,,\, \cdot \,)$ of $H_0$ continuously extends to $z \in \ol{\bbC_+}$. In addition, in the massless case $m=0$, the limit $z \to 0$ exists\footnote{Our choice of notation $0+i\,0$ in 
$G_0(0+i\,0;x,y)$ indicates that the limit $\lim_{z \to 0}$ is performed in the closed upper half-plane 
$\ol{\bbC_+}$.}, 
\begin{align}
\begin{split} 
& \lim_{\substack{z \to 0, \\ z \in \ol{\bbC_+}\backslash\{0\}}} G_0(z;x,y) := G_0(0+i\,0;x,y)   \\
& \quad = i 2^{-1} \pi^{-n/2} \Gamma(n/2) \, \alpha \cdot \f{(x - y)}{|x - y|^n}, \quad x, y \in \bbR^n, \, x \neq y, \; n \in \bbN, \, n \geq 2,    \lb{5.18} 
\end{split} 
\end{align} 
and no blow up occurs for all $n \in \bbN$, $n \geq 2$. This observation is consistent with the sufficient condition for the Dirac operator $H = H_0 + V$ (in dimensions $n \in \bbN$, $n \geq 2$), with $V$ an appropriate self-adjoint $N \times N$ matrix-valued potential, having no eigenvalues, as derived in \cite[Theorems~2.1, 2.3]{KOY15}. 

The following remark is primarily of a heuristic nature; at this point it just serves as a motivation for our detailed discussion of absence of threshold resonances for Schr\"odinger and Dirac operators in Section \ref{s3}. 
 
\begin{remark} \lb{r5.1} 
The asymptotic behavior, for some $d_n \in (0,\infty)$, 
\begin{align}
\|G_0(0+i\,0;x,y)\|_{\cB(\bbC^N)} \underset{\substack{z \to 0, \\ z \in \ol{\bbC_+}\backslash\{0\}}}{=}  
d_n |x-y|^{1-n}, \quad x, y \in \bbR^n, \, x \neq y, \; n \in \bbN, \, n \geq 2, 
\end{align} 
combined with the homogeneous Lippmann--Schwinger-type integral equation in \eqref{9.37a}, formally ``implies'' the absence of zero-energy resonances (cf.\ Section \ref{s3} for a detailed discussion) of $H$ for $n\in \bbN$, $n\geq 3$, for sufficiently fast decaying short-range potentials $V$ at infinity, as $|\, \cdot \,|^{1-n}$ lies in 
$L^2(\bbR^n; d^n x)$ near infinity if and only if $n \geq 3$. This is consistent with observations in \cite{Ai16}, \cite[Sect.~4.4]{BE11}, \cite{BES08}, \cite{BGW95}, \cite{SU08}, \cite{SU08a}, \cite{ZG13} for $n=3$ (see also Remark \ref{r9.8}\,$(iii)$). 
This should be contrasted with the behavior of Schr\"odinger operators where 
\begin{align}
\begin{split}
\lim_{\substack{z \to 0 \\ z \in \bbC \backslash \{0\}}} g_0(z;x,y) = g_0(0;x,y) 
= \f{1}{(n-2) \omega_{n-1}} |x - y|^{2-n},& \\ 
 x, y \in\bbR^n, \, x \neq y, \; n \in \bbN, \; n \ge 3,&    
 \end{split} 
\end{align} 
``implies'' the absence of zero-energy resonances of $h= h_0 + v$ for $n\in \bbN$, $n\geq 5$, again for sufficiently fast decaying short-range potentials $v$ at infinity, as $|\, \cdot \,|^{2-n}$ lies in $L^2(\bbR^n; d^n x)$ near infinity if and only if $n \geq 5$, as observed in \cite{Je80}.

Since $H_0$ has no spectral gap, $\sigma(H_0) = \bbR$, but $h_0$ has the half-line 
$(-\infty, 0)$ in its resolvent set, a comparison of $h_0$ with the massive free 
Dirac operator $H_0(m) = H_0 + m \, \beta$, $m > 0$, with spectral gap $(- m, m)$, replacing the energy $z=0$ by $z=\pm m$, is quite natural and then exhibits a similar logarithmic blowup behavior as $z \to 0$
in dimensions $n=2$. Finally, the leading asymptotic behavior $|\, \cdot \,|^{2-n}$ of the free massive Green's matrix $G_0(m,\pm m;\dott,0)$ in \eqref{5.14} near infinity then ``implies'' the absence of threshold resonances of $H(m) = H_0(m) + V$ for $n\in \bbN$, $n\geq 5$, as in the case of Schr\"odinger operators $h$. 

Section \ref{s3} is devoted to a rigorous treatment of the formal statements in this remark. \hfill $\diamond$
\end{remark}

Next, we also recall some basic facts on $L^p$-properties of Riesz potentials (see, e.g., 
\cite[Sect.~V.1]{St70}):
\begin{theorem} \lb{t5.4A}
Let $n \in \bbN$, $\alpha \in (0,n)$, and introduce the Riesz potential operator $\cR_{\alpha,n}$ as follows: 
\begin{align}
\begin{split} 
& (\cR_{\alpha,n} f)(x) = \big((- \Delta)^{- \alpha/2} f\big)(x) 
= \gamma(\alpha,n)^{-1} \int_{\bbR^n} d^n y \, |x - y|^{\alpha - n} f(y),     \lb{5.50} \\
& \gamma(\alpha,n) = \pi^{n/2} 2^{\alpha} \Gamma(\alpha/2)/\Gamma((n-\alpha)/2), 
\end{split}
\end{align}
for appropriate functions $f$ $($see below\,$)$. \\[1mm] 
$(i)$ Let $p \in [1,\infty)$ and $f \in L^p(\bbR^n)$. Then the integral $(\cR_{\alpha,n} f)(x)$ converges for $($Lebesgue\,$)$ a.e.~$x \in \bbR^n$.   \\[1mm] 
$(ii)$ Let $1 < p < q < \infty$, $q^{-1} = p^{-1} - \alpha n^{-1}$, and $f \in L^p(\bbR^n)$. Then there 
exists $C_{p,q,\alpha,n} \in (0,\infty)$ such that
\begin{equation}
\|\cR_{\alpha,n} f\|_{L^q(\bbR^n)} 
\leq C_{p,q,\alpha,n} \|f\|_{L^p(\bbR^n)}.    \lb{5.51} 
\end{equation}
\end{theorem}

We also note the $\beta$ function-type integral (cf.\ \cite[p.~118]{St70}),
\begin{align}
& \int_{\bbR^n} d^n y \, |e_k - y|^{\alpha - n} |y|^{\beta - n} 
= \gamma(\alpha,n) \gamma(\beta,n)/\gamma(\alpha + \beta,n),    \no \\
& 0 < \alpha < n, \; 0 < \beta < n, \; \alpha + \beta < n,    \lb{5.52} \\
& e_k =(0,\dots,\underbrace{1}_{k},\dots,0), \; 1 \leq k \leq n,    \no 
\end{align}
and more generally, the Riesz composition formula (see \cite[Sects.~3.1, 3.2]{Du70}),
\begin{align}
\begin{split}
\int_{\bbR^n} d^n y \, |x - y|^{\alpha - n} |y - w|^{\beta - n} 
= [\gamma(\alpha,n) \gamma(\beta,n)/\gamma(\alpha + \beta,n)] |x - w|^{\alpha + \beta -n},& \\
\quad 0 < \alpha < n, \; 0 < \beta < n, \; 0 < \alpha + \beta < n, \; x, w \in \bbR^n.    \lb{5.52A} 
\end{split} 
\end{align}

For later use in Section \ref{s3}, we recall the following estimate taken from \cite[Lemma~6.3]{EG10}.
\begin{lemma} \lb{l3.12}
Let $n\in \bbN$ and $x_1,x_2\in \bbR^n$.  If $k,\ell\in [0,n)$, $\varepsilon, \beta \in (0,\infty)$, with $k+\ell+\beta \geq n$, and $k+\ell\neq n$, then
\begin{align}
\begin{split}
&\int_{\bbR^n} d^ny \, |x_1 - y|^{-k} \langle y\rangle^{-\beta-\varepsilon} |y - x_2|^{-\ell}     \\
&\quad \leq C_{n,k,\ell,\beta,\varepsilon} \cdot
\begin{cases}
|x_1 - x_2|^{-\max\{0,k+\ell-n\}},& |x_1 - x_2|\leq 1,\\
|x_1 - x_2|^{-\min\{k,\ell,k+\ell+\beta-n\}},& |x_1 - x_2| \geq 1,
\end{cases}
\end{split} 
\end{align}
where $C_{n,k,\ell,\beta,\varepsilon}\in (0,\infty)$ is an $x_1, x_2$-independent constant.
\end{lemma}

We conclude this section by recalling the following interesting results of McOwen \cite{Mc79} and Nirenberg--Walker \cite{NW73}, which provide necessary and sufficient conditions for the boundedness of certain classes of integral operators in $L^p(\bbR^n)$:

\begin{theorem} \lb{t5.6}
Let $n \in \bbN$, $c, d \in \bbR$, $c + d > 0$, $p \in (1,\infty)$, and $p'=p/(p-1)$.  The following items $(i)$ and $(ii)$ hold.\\
$(i)$  If
\begin{equation}
K_{c,d}(x,y) = |x|^{-c} |x - y|^{(c + d) - n} |y|^{-d}, 
\quad x, y \in \bbR^n, \; x \neq x',    \lb{5.21} 
\end{equation}
then the integral operator $K_{c,d}$ in $L^p(\bbR^n)$ with integral kernel 
$K_{c,d}(\, \cdot \,, \, \cdot \,)$ 
in \eqref{5.21} is bounded if and only if $c < n/p$ and $d < n/p'$.\\[1mm] 
$(ii)$ If
\begin{equation}
\wti K_{c,d}(x,y) = (1+|x|)^{-c} |x - y|^{(c + d) - n} (1+|y|)^{-d}, 
\quad x, y \in \bbR^n, \; x \neq x',    \lb{5.21mc}
\end{equation}
then the integral operator $\wti K_{c,d}$ in $L^p(\bbR^n)$ with integral kernel 
$\wti K_{c,d}(\, \cdot \,, \, \cdot \,)$ 
in \eqref{5.21mc} is bounded if and only if $c < n/p$ and $d < n/p'$.
\end{theorem}

\section{Nonexistence of Threshold Resonances} \lb{s3}

In this section we prove our principal results on nonexistence of threshold resonances in three cases: First, 
in the case of Schr\"odinger operators in dimension $n \geq 3$; second, in the case of massless Dirac 
operators in dimensions $n \geq 2$; and third, in the case of massive Dirac operators in dimensions 
$n \geq 3$. \\[1mm] 
\noindent 
{\bf Schr\"odinger Operators in $\boldsymbol{\bbR^n}$, $\boldsymbol{n \geq 3}$.} 
We begin with the case of Schr\"odinger operators $h$ in dimension $n \geq 3$ as defined 
in \eqref{2.14}, and start by making the following assumptions on the potential $v$.    

\begin{hypothesis} \lb{h9.5a}
Let $n \in \bbN$, $n \geq 3$. Assume the a.e.~real-valued potential 
$v$ satisfies for some fixed $C \in (0,\infty)$, 
\begin{equation}
v \in L^{\infty} (\bbR^n), \quad |v(x)| \leq C \langle x \rangle^{- 4} \, 
\text{ for a.e.~$x \in \bbR^n$.}    \lb{9.31Aa}
\end{equation}
In addition, we suppose that 
\begin{align} 
\begin{split}
& v = v_1 v_2 = |v|^{1/2} u_v |v|^{1/2},   \\
& \text{where } \, v_1 = |v|^{1/2}, \quad 
v_2 = u_v |v|^{1/2}, \quad u_v = \sgn(v). 
\end{split} 
\end{align} 
\end{hypothesis} 

Here we abbreviated 
\begin{equation}
\sgn(v(x)) = \begin{cases} 1, & v(x) \geq 0, \\
-1, & v(x) < 0, \end{cases}\, \text{ for a.e.~$x\in \bbR^n$.}
\end{equation}

We continue with the threshold behavior, that is, the $z=0$ behavior, of $h$: 

\begin{definition} \lb{d9.6a} Assume Hypothesis \ref{h9.5a}. \\[1mm] 
$(i)$ The point $0$ is called a zero-energy eigenvalue of $h$ if $h \psi = 0$ has a distributional solution 
$\psi$ satisfying 
\begin{equation} 
\psi \in \dom(h) = W^{2,2}(\bbR^n)    
\end{equation}
$($equivalently, $\ker(h) \supsetneqq \{0\}$$)$. \\[1mm] 
$(ii)$ The point $0$ is called a zero-energy $($or threshold\,$)$ resonance of $h$ if 
\begin{equation}
\ker\big(\big[I_{L^2(\bbR^n)} + \ol{v_2 (h_0 + 0 I_{L^2(\bbR^n)})^{-1} v_1}\big]\big) \supsetneqq \{0\}  
\end{equation}
and if there exists $0 \neq \phi \in \ker\big(\big[I_{L^2(\bbR^n)} + \ol{v_2 (h_0 + 0 I_{L^2(\bbR^n)}^{-1} v_1}\big]\big)$ such that $\psi$ defined by 
\begin{align}
\begin{split} 
& \psi(x) = - \big((h_0 + 0 I_{L^2(\bbR^n)})^{-1} v_1 \phi\big)(x)     \lb{9.46Aa} \\
& \hspace*{8.5mm} = [(n-2) \omega_{n-1}]^{-1} \int_{\bbR^n} d^n y \, |x-y|^{2-n} v_1(y) \phi(y)     
\end{split} 
\end{align}
$($for a.e.~$x \in \bbR^n$, $n \geq 3$$)$ is a distributional solution of $h \psi = 0$ satisfying 
\begin{equation} 
\psi \notin L^2(\bbR^n).
\end{equation}
$(iii)$ $0$ is called a regular point for $h$ if it is neither a zero-energy eigenvalue nor a zero-energy 
resonance of $h$.  
\end{definition}

Additional properties of $\psi$ are isolated in Theorem \ref{t9.7a}.

While the point $0$ being regular for $h$ is the generic situation, zero-energy eigenvalues and/or resonances are exceptional cases. 

Next, we introduce the following convenient abbreviation (for $x, y \in \bbR^n$, $x \neq y$):
\begin{align}
\begin{split} 
r_{0,0}(x-y) &= \lim_{\substack{z \to 0 \\ z \in \bbC \backslash \{0\}}} g(z;x,y) =
g_0(0;x,y) = [(n-2) \omega_{n-1}]^{-1} |x - y|^{2-n}      \\
&= 4^{-1} \pi^{- n/2} \Gamma((n-2)/2) |x - y|^{2-n}, \quad n \geq 3,   \lb{9.46a} 
\end{split} 
\end{align}
and note that
\begin{align}
- \Delta_x r_{0,0}(x-y) = - \Delta_x g_{0}(0;x,y) = \delta(x-y),  \quad 
x, y \in \bbR^n, \; x \neq y, \; n \geq 3,    &   \lb{3.9A} 
\end{align}
in the sense of distributions. 

In the remainder of this section we will frequently apply 
\cite[Prop.~6.10]{Fo99}, that is, the fact that $0 < p < q < r \leq \infty$ implies $L^p \cap L^r \subset L^q$. 

\begin{theorem} \lb{t9.7a}
Assume Hypothesis \ref{h9.5a}.  \\[1mm] 
$(i)$ If $n=3,4$, there are precisely four possible cases: \\[1mm]
Case $(I)$: $0$ is regular for $h$. \\[1mm]
Case $(II)$: $0$ is a $($necessarily simple\footnote{One can show that if $n=3,4$, there is at most one 
resonance function in Case $(II)$, see \cite{AGH82}, \cite{Je84}, \cite{JK79}.}\,$)$ resonance of $h$. 
In this case, the resonance function $\psi$ satisfies 
\begin{align} 
& \psi \in L^q(\bbR^3), \quad q \in (3, \infty) \cup\{\infty\},   \\
&  \psi \in L^q(\bbR^4), \quad q \in (2, \infty) \cup\{\infty\},   \\
& \nabla \psi \in \big[L^2(\bbR^n)\big]^n, \quad \Delta \psi \in L^q(\bbR^n), \; q \in [2,\infty) \cup \{\infty\}, \; n =3,4,   \\
&  \psi \notin L^2(\bbR^n), \quad n = 3,4.
\end{align}
Case $(III)$: $0$ is a $($possibly degenerate\,$)$ eigenvalue of $h$. In this case, the corresponding 
eigenfunctions $\psi \in \dom(h) = W^{2,2}(\bbR^n)$, $n \in \bbN$, $n \geq 3$, of $h \psi = 0$ also satisfy 
\begin{align} 
& \psi, \Delta \psi \in L^q(\bbR^n), \quad q \in [2, \infty) \cup\{\infty\}, \;  n =3,4,   \\
& \nabla \psi \in \big[L^2(\bbR^n)\big]^n, \;  n =3,4.   
\end{align}
Case $(IV)$: A possible mixture of Cases $(II)$ and $(III)$. \\[1mm]
$(ii)$ If $n \in \bbN$, $n \geq 5$, there are precisely two possible cases: \\[1mm]
Case $(I)$: $0$ is regular for $h$. \\[1mm]
Case $(II)$: $0$ is a $($possibly degenerate\,$)$ eigenvalue of $h$. In this case, the corresponding eigenfunctions $\psi \in \dom(h) = W^{2,2}(\bbR^n)$ of $h \psi = 0$ also satisfy
\begin{align} 
& \psi, \, \Delta \psi \in L^q(\bbR^n), \quad q \in \begin{cases} (n/(n-2), \infty) \cup \{\infty\}, & 5 \leq n \leq 7, \\
(2n/(n+4), 2n/(n-4)), & n \geq 8, 
\end{cases}     \lb{3.16A} \\ 
& \nabla \psi \in \big[L^q(\bbR^n)\big]^n, \quad q \in \begin{cases} (5/3,\infty), & n = 5, \\
(3/2,\infty), & n=6, \\
(7/5,14), & n=7, \\
(2n/(n+4),2n/(n-6)), & n \geq 8.
\end{cases}      \lb{3.17A} 
\end{align}
In particular, there are no zero-energy resonances of $h$ in dimension $n \geq 5$. \\[1mm] 
$(iii)$ The point $0$ is regular for $h$ if and only if 
\begin{equation}
\ker\big(\big[I_{L^2(\bbR^n)} + \ol{v_2 (h_0 + 0 I_{L^2(\bbR^n)})^{-1} v_1}\big]\big) = \{0\}. 
\end{equation}
\end{theorem}
\begin{proof}
Since $g_0(0;x,y)$, $x\neq y$, exists for all $n \geq 3$ (cf.\ \eqref{5.3}), the 
Birman--Schwinger eigenvalue equation  
\begin{equation} 
\big[I_{L^2(\bbR^n)} + \ol{v_2 (h_0 - 0 I_{L^2(\bbR^n)})^{-1} v_1}\big] \phi_0 = 0,  \quad 
0 \neq \phi_0 \in L^2(\bbR^n),   \lb{9.36A} 
\end{equation}
gives rise to a distributional zero-energy solution $\psi_0 \in L^1_{\loc}(\bbR^n)$ of $h \psi_0 = 0$ in 
terms of $\phi_0$ of the form (for a.e.~$x \in \bbR^n$, $n \geq 3$),
\begin{align}
& \psi_0(x) = - \big((h_0 - 0 I_{L^2(\bbR^n)})^{-1} v_1 \phi_0\big)(x)     \no \\
& \hspace*{8.5mm} = - [r_{0,0} * (v_1 \phi_0)](x)    \lb{9.36aA} \\
& \hspace*{8.5mm} = - 4^{-1} \pi^{- n/2} \Gamma((n-2)/2) \int_{\bbR^n} d^n y \, |x-y|^{2-n} 
v_1(y) \phi_0(y),   \lb{9.37A} \\
& \phi_0(x) = (v_2 \psi_0)(x).       \lb{9.38A}
\end{align}
In particular, one concludes that $\psi_0 \neq 0$. 

We note from the outset that $\Delta \psi = v \psi$ with $v \in L^{\infty}(\bbR^n)$ yields 
$\Delta \psi \in L^r(\bbR^n)$ whenever $\psi \in L^r(\bbR^n)$ for some $r \geq 1$. Moreover, 
since also $\Delta \psi = v_1 \phi$, with $v_1 \in L^{\infty}(\bbR^n)$ and $\phi \in L^2(\bbR^n)$, this  
yields $\Delta \psi \in L^2(\bbR^n)$, and, if in addition $\psi$ is an eigenfunction of $H$, then of course 
$\psi \in W^{1,2}(\bbR^n)$. 

Thus, one estimates, 
with $|v_1(y)| \leq c \langle y \rangle^{-2}$ and some constant $d_n \in (0, \infty)$, 
\begin{align}
|\psi_0(x)| &\leq d_n \int_{\bbR^n} d^n y \, |x - y|^{2-n} \langle y \rangle^{-2} |\phi_0(y)|  \no \\
& = d_n \cR_{2,n}\big(\langle \dott \rangle^{-2} |\phi_0(\dott)|\big)(x)   \lb{9.41A} \\ 
& = d_n \cR_{2,n}\big(\langle \dott \rangle^{-2} |v_2(\dott)| |\psi_0(\dott)|\big)(x),   \lb{9.41C} 
\quad x \in \bbR^n.  
\end{align} 
An application of Theorem \ref{t5.6}\,$(ii)$ with $c=0$, $d=2$, $p=p'=2$, and the inequality $2 < n/2$, combined with 
$\phi_0 \in L^2(\bbR^n)$, then yield 
\begin{equation}
 \psi_0 \in L^2(\bbR^n), \quad n \geq 5.     \lb{9.42A} 
\end{equation} 
To prove that actually $\psi_0 \in \dom(h) = W^{2,2}(\bbR^n)$, it suffices to argue as follows:  
\begin{align}
\Delta \psi_0 = v \psi_0 \in L^2(\bbR^n), \quad n \geq 5,     \lb{9.43A}
\end{align}     
in the sense of distributions since $v \in L^{\infty} (\bbR^n)$ and $\psi_0 \in L^2(\bbR^n)$. Thus, 
also $\nabla \psi_0 \in [L^2(\bbR^n)]^n$, and hence, 
\begin{equation}
\psi_0 \in W^{2,2}(\bbR^n), \quad n \geq 5.     \lb{9.45A}
\end{equation}

Moreover, employing \eqref{9.36aA}, one obtains for $n \geq 3$ that 
\begin{equation}
- (\Delta \psi_0)(x) = \Delta_x (r_{0,0} * v_1 \phi_0)(x) = - v_1(x) \phi_0(x), 
\end{equation}
and hence
\begin{equation}
(\Delta \psi_0) \in L^2(\bbR^n), \quad n \geq 3, 
\end{equation}
since $v_1 \in L^{\infty}(\bbR^n)$ and $\phi_0 \in L^2(\bbR^n)$. In addition, \eqref{9.38A} yields for 
a.e.~$x \in \bbR^n$, $n \geq 3$,
\begin{equation}\lb{Zz3.31}
(\nabla \psi_0)(x) = 4^{-1} \pi^{- n/2} (2-n) \Gamma((n-2)/2) \int_{\bbR^n} d^n y \, |x-y|^{1-n} \f{(x-y)}{|x-y|}
v_1(y) \phi_0(y), 
\end{equation}
implying
\begin{align}
|(\nabla \psi_0)(x)| &\leq \wti c_n  \int_{\bbR^n} d^n y \, |x-y|^{1-n} |v_1(y)| |\phi_0(y)|      \no \\
& \leq \widehat c_n \int_{\bbR^n} d^n y \, |x-y|^{1-n} \langle y \rangle^{-2} |\phi_0(y)|    \lb{3.31A}
\end{align}
for some $\wti c_n, \widehat c_n \in (0,\infty)$. An application of Theorem \ref{t5.6}\,$(ii)$ with $c=0$, $d=1$, 
$p=p'=2$, and the condition $ 1 < n/2$ then implies 
\begin{equation}
(\nabla \psi_0) \in \big[L^2(\bbR^n)\big]^n, \quad n \geq 3. 
\end{equation}

Returning to the estimate \eqref{9.41A}, and invoking the Riesz potential $\cR_{2,n}$ 
(cf.\ Theorem \ref{t5.4A}), 
one obtains (for some constant $\wti d_n \in (0,\infty)$)
\begin{align}
|\psi_0(x)| & \leq d_n \int_{\bbR^n} d^n y \, |x - y|^{2-n} \langle y \rangle^{-2} |\phi_0(y)|  \no \\
& \leq \wti d_n \cR_{2,n} \big(\langle y \rangle^{-2} |\phi_0(\dott)|\big)(x),  \quad x \in \bbR^n, 
\end{align}
and hence \eqref{5.51} implies (for some constant $\wti C_{p,q,n} \in (0,\infty)$)  
\begin{align}
\|\psi_0\|_{L^q(\bbR^n)} & \leq \wti d_n \big\|\cR_{2,n} \big(\langle y \rangle^{-2} |\phi_0(\dott)|\big)
\big\|_{L^q(\bbR^n)}    \no \\
& \leq \wti C_{p,q,n} \big\|\langle y \rangle^{-2} |\phi_0(\dott)|\big\|_{L^p(\bbR^n)}   \no \\
& \leq \wti C_{p,q,n} \big\|\langle y \rangle^{-2}\big\|_{L^s(\bbR^n)} \|\phi_0\|\|_{L^2(\bbR^n)},   \lb{9.61A} \\
& \hspace*{-2.75cm} 1 < p < q < \infty, \; p^{-1} = q^{-1} + 2n^{-1}, \; s = 2qn [2n + 4q - qn]^{-1} \geq 1.  \no
\end{align}
In particular, 
\begin{equation}
p = qn/(n + 2q), \quad 2n + 4q - qn > 0. 
\end{equation}

\noindent 
$(a)$ The case $n=3$: Then $p = 3q/(3 + 2q)$ and $p >1$ requires $q \in (3,\infty)$, hence  
$p \in (1,3/2)$, and $s = 6q/(q + 6) > 2$. Thus, \eqref{9.61A} and 
$\big\|\langle \dott \rangle^{-2}\big\|_{L^s(\bbR^3)} < \infty$ 
imply
\begin{equation}
\psi_0 \in L^q(\bbR^3), \quad q \in (3,\infty). 
\end{equation} 

To prove that a resonance function (resp., eigenfunction) satisfies $\psi\in L^{\infty}(\bbR^3)$, one applies \eqref{9.38A} and the fact that $|v_1(\dott)|\leq c\langle\dott\rangle^{-2}$ for some constant $c\in (0,\infty)$ in \eqref{9.37A} to obtain
\begin{equation}\lb{3.38c}
|\psi(x)|\leq \widetilde{d}_3\int_{\bbR^3}d^3y\, |x-y|^{-1}\langle y\rangle^{-4}|\psi(y)|,\quad x\in \bbR^3,
\end{equation}
for an appropriate $x$-independent constant $\widetilde{d}_3\in (0,\infty)$.  By H\"older's inequality (with conjugate exponents $p=5/4$ and $p'=5$),
\begin{align}
|\psi(x)|\leq \widetilde{d}_3\bigg(\int_{\bbR^3}d^3y\, |x-y|^{-5/4}\langle y\rangle^{-5}\bigg)^{4/5}\bigg(\int_{\bbR^3}d^3y\, |\psi(y)|^5\bigg)^{1/5},\quad x\in \bbR^3.\lb{3.39c}
\end{align}
The second integral on the right-hand side in \eqref{3.39c} is finite since a resonance function (resp., eigenfunction) satisfies $\psi\in L^5(\bbR^3)$.  An application of Lemma \ref{l3.12} with $x_1 = x$, $k=5/4$, $\ell=0$, $\beta=4$, and $\varepsilon =1$ then implies
\begin{equation}
|\psi(x)|\leq \widetilde{d}_3 C_{3,5/4,0,4,1}^{4/5}\|\psi\|_{L^5(\bbR^3)},\quad x\in \bbR^3,
\end{equation}
and hence it follows that $\psi\in L^{\infty}(\bbR^3)$. 

\noindent 
$(b)$ The case $n=4$: Then $p = 4q/(4 + 2q)$ and $p >1$ requires $q \in (2,\infty)$, hence  
$p \in (1,2)$, and $s = q > 2$. Thus, \eqref{9.61A} and $\big\|(1 + |\dott|)^{-2}\big\|_{L^s(\bbR^4)} < \infty$ 
implies
\begin{equation}
\psi_0 \in L^q(\bbR^4), \quad q \in (2,\infty). 
\end{equation} 

To prove that a resonance function (resp., eigenfunction) satisfies $\psi\in L^{\infty}(\bbR^4)$, one applies \eqref{9.38A} and the fact that $|v_1(\dott)|\leq c\langle\dott\rangle^{-2}$ for some constant $c\in (0,\infty)$ in \eqref{9.37A} to obtain
\begin{equation}\lb{3.42c}
|\psi(x)|\leq \widetilde{d}_4\int_{\bbR^4}d^4y\, |x-y|^{-2}\langle y\rangle^{-4}|\psi(y)|,\quad x\in \bbR^4,
\end{equation}
for an appropriate $x$-independent constant $\widetilde{d}_4\in (0,\infty)$.  By H\"older's inequality (with conjugate exponents $p=3/2$ and $p'=3$),
\begin{align}
|\psi(x)|\leq \widetilde{d}_4\bigg(\int_{\bbR^4}d^4y\, |x-y|^{-3}\langle y\rangle^{-6}\bigg)^{2/3}\bigg(\int_{\bbR^4}d^4y\, |\psi(y)|^3\bigg)^{1/3},\quad x\in \bbR^4.\lb{3.43c}
\end{align}
The second integral on the right-hand side in \eqref{3.43c} is finite since a resonance function (resp., eigenfunction) satisfies $\psi\in L^3(\bbR^4)$.  An application of Lemma \ref{l3.12} with $x_1=x$, $k=3$, $\ell=0$, 
$\beta=5$, and $\varepsilon = 1$ then implies
\begin{equation}
|\psi(x)|\leq \widetilde{d}_4 C_{4,3,0,5,1}^{2/3}\|\psi\|_{L^3(\bbR^4)},\quad x\in \bbR^4,
\end{equation}
and it follows that $\psi\in L^{\infty}(\bbR^4)$. 

\noindent 
$(c)$ The case $n \geq 5$: Returning to \eqref{9.37A}, we employ the fact $\phi_0 = v_2 \psi_0$ and hence 
obtain for some constants $\wti D_n, D_n \in (0,\infty)$, 
\begin{align}
|\psi_0(x)| &\leq \wti D_n \int_{\bbR^n} d^n y \, |x-y|^{2-n} \langle y \rangle^{-4} |\psi_0(y)|    \no \\
&= D_n \cR_{2,n} \big(\langle \dott \rangle^{-4} |\psi_0(\dott)|\big)(x),\quad x\in\bbR^n,
\end{align}
and hence by Theorem \ref{t5.4A}\,$(ii)$  
\begin{align}
\|\psi_0\|_{L^q(\bbR^n)} 
&\leq D_n \big\|\cR_{2,n} \big(\langle \dott \rangle^{-4} |\psi_0(\dott)|\big)\big\|_{L^q(\bbR^n)}  \no \\
& \leq \widehat C_{p,q,n} \big\|\langle \dott \rangle^{-4} |\psi_0(\dott)|\big\|_{L^p(\bbR^n)}    \no \\
& \leq \widehat C_{p,q,n} \big\|\langle \dott \rangle^{-4}\big\|_{L^s(\bbR^n)} \|\psi_0\|_{L^2(\bbR^n)}
\end{align}
for some $\widehat C_{p,q,n} \in (0,\infty)$. Then $p = qn/(n + 2q) > 1$ yields $q > n/(n - 2)$. Since 
$s = 2qn/(2n + 4q - qn)$ and one needs $s \geq 1$ and hence $2n + 4q - qn > 0$, this implies $q < 2n/(n - 4)$. Thus, 
$q \in (n/(n-2), 2n/(n-4))$ and hence $s > 2$. The condition 
$\big\|\langle \dott \rangle^{-4}\big\|_{L^s(\bbR^n)} < \infty$ then implies $q > 2n/(n+4)$. Altogether one obtains 
\begin{equation}
q \in \begin{cases} (n/(n-2), 2n/(n-4)), & 5 \leq n \leq 7, \\
(2n/(n+4), 2n/(n-4)), & n \geq 8, 
\end{cases}
\end{equation} 
and thus, 
\begin{equation}
\psi_0 \in L^q(\bbR^n), \quad q \in \begin{cases} (n/(n-2), 2n/(n-4)), & 5 \leq n \leq 7, \\
(2n/(n+4), 2n/(n-4)), & n \geq 8.   \lb{3.49A} 
\end{cases}
\end{equation}
Since $\Delta \psi_0 = v \psi_0$ and $v \in L^{\infty}(\bbR^n)$, \eqref{3.49A} also implies
\begin{equation}
\Delta \psi_0 \in L^q(\bbR^n), \quad q \in \begin{cases} (n/(n-2), 2n/(n-4)), & 5 \leq n \leq 7, \\
(2n/(n+4), 2n/(n-4)), & n \geq 8, 
\end{cases}
\end{equation}
and hence \eqref{3.16A}.

Employing $\phi_0 = v_2 \psi_0$ in \eqref{3.31A} yields 
\begin{align}
|(\nabla \psi_0)(x)| & \leq \wti C_n \int_{\bbR^n} d^n y \, |x-y|^{1-n} \langle y \rangle^{-4} |\psi_0(y)|  \no \\
& = C_n \cR_{1,n}\big(\langle \dott \rangle^{-4} |\psi_0(\dott)|\big)(x), \quad x\in\bbR^n,    \lb{3.51A} 
\end{align}
for some $\wti C_n, C_n \in (0,\infty)$. Thus, Theorem \ref{t5.4A}\,$(ii)$ yields
\begin{equation}
|\nabla \psi_0| \in L^r(\bbR^n), \quad r = qn/(n-q), \; q < n, 
\end{equation}
with $q$ given as in \eqref{3.49A}. Working out the details yields 
\begin{equation}
|\nabla \psi_0| \in L^r(\bbR^n), \quad r \in \begin{cases} (5/2,\infty), & n=5, \\
(2,\infty), & n=6, \\
(7/4,14), & n=7, \\
(2n/(n-2), 2n/(n-6)), & n \geq 8.
\end{cases}     \lb{3.53A}
\end{equation}
On the other hand, applying Theorem \ref{t5.6}\,$(ii)$ to the first line in \eqref{3.51A} with $c=0$, $d=1$, 
$1 < n/q' = n(q-1)/q$, and hence $q > n/(n-1)$ yields 
 \begin{equation}
|\nabla \psi_0| \in L^q(\bbR^n), \quad  q \in \begin{cases} (n/(n-2), 2n/(n-4)), & 5 \leq n \leq 7, \\
(2n/(n+4), 2n/(n-4)), & n \geq 8. 
\end{cases}     \lb{3.54A} 
\end{equation}
A comparison of \eqref{3.53A} and \eqref{3.54A} implies \eqref{3.17A}. 

To prove essential boundedness of eigenfunctions $\psi$ in dimensions $n \in \{5, 6, 7\}$, one notes that a combination of \eqref{9.37A}, \eqref{9.38A}, and the fact that $|v_1(\dott)|\leq c\langle \dott\rangle^{-2}$ for some constant $c\in (0,\infty)$, yields
\begin{align}\lb{3.54c}
|\psi(x)|\leq \widetilde{d}_n\int_{\bbR^n}d^ny\, |x-y|^{2-n}\langle y\rangle^{-4}|\psi(y)|,\quad x\in \bbR^n,\, n\geq 5,
\end{align}
for an appropriate $x$-independent constant $\widetilde{d}_n\in (0,\infty)$.  To prove that the right-hand side in \eqref{3.54c} is essentially bounded, we make use of \eqref{3.49A} and consider the cases $n\in \{5,6,7\}$ next.
By \eqref{3.49A}, $\psi\in L^q(\bbR^n)$ for all
\begin{equation}\lb{3.55c}
q\in (n/(n-2),2n/(n-4)).
\end{equation}
Note that $q$ satisfies \eqref{3.55c} if and only if its conjugate exponent $q'=q/(q-1)$ satisfies
\begin{equation}\lb{3.56c}
q' \in (2n/(n+4),n/2).
\end{equation}
Choose $q_0' \in(2n/(n+4),n/(n-2))\subset (2n/(n+4),n/2)$.  Then $q_0:=q_0'/(q_0'-1)\in (n/(n-2),2n/(n-4))$ and $\psi\in L^{q_0}(\bbR^n)$.  Applying H\"older's inequality with conjugate exponents $q_0'$ and $q_0$ on the right-hand side of \eqref{3.54c}, one obtains
\begin{align}
|\psi(x)|\leq \widetilde{d}_n \bigg(\int_{\bbR^n}d^ny\, |x-y|^{(2-n)q_0'}\langle y\rangle^{-4q_0'} 
\bigg)^{1/q_0'}\bigg(\int_{\bbR^n}d^ny\,|\psi(y)|^{q_0} \bigg)^{1/q_0},\quad x\in \bbR^n.\lb{3.57c}
\end{align}
The second integral on the right-hand side in \eqref{3.57c} if finite since $\psi\in L^{q_0}(\bbR^n)$.  To estimate the first integral on the right-hand side in \eqref{3.57c}, one applies Lemma \ref{l3.12} with the choices $x_1=x$, $k=k_0:=(n-2)q_0'$, $\ell=0$, $\beta=\beta_0:= 4 q_0' - [4n/(n+4)]$, and 
$\varepsilon = \varepsilon_0 := 4n/(n+4)$, to obtain
\begin{equation}\lb{3.58c}
|\psi(x)|\leq \widetilde{d}_n C_{n,k_0,0,\beta_0, \varepsilon_0}^{1/q_0'}\|\psi\|_{L^{q_0}(\bbR^n)}, 
\quad x\in \bbR^n.
\end{equation}
One notes that the hypotheses of Lemma \ref{l3.12} are satisfies since $k=(n-2)q_0' < (n-2)\cdot n/(n-2) = n$, while
\begin{equation}
\beta = 4q_0' - 4n/(n+4) \geq 8n/(n+4) - 4n/(n+4) > 0, \quad \langle\dott\rangle^{-4q_0'} 
= \langle\dott\rangle^{-\beta_0-4n/(n+4)},
\end{equation}
and
\begin{equation}
k+\ell+\beta = n q_0' + 2 q_0' - 4n/(n+4) \geq n q_0' + 4n/(n+4) - 4n/(n+4) = n q_0' > n.
\end{equation}
The inequality in \eqref{3.58c} implies $\psi\in L^{\infty}(\bbR^n)$ for $5 \leq n \leq  7$.
In Remark \ref{rA.1} we will illustrate why the same line of reasoning fails for $n \geq 8$.

Finally, we will prove that if $\ker(h) \supsetneqq \{0\}$ then one necessarily also obtains that   
$\ker \big(\big[I_{L^2(\bbR^n)} + \ol{v_2 (h_0 + 0 I_{L^2(\bbR^n)})^{-1} v_1}\big]\big) 
\supsetneqq \{0\}$. Indeed, if $0 \neq \psi_0 \in \ker(h)$, then 
$\phi_0 := v_2 \psi_0 = u_v v_1 \psi_0 \in L^2(\bbR^n)$ and hence 
$v_1 \phi_0 \in L^2(\bbR^n)$. Then, $h \psi_0 = 0$ yields 
$\Delta \psi_0 = v \psi_0 = v_1 v_2 \psi_0 = v_1 \phi_0$. 

Thus, an application of \eqref{3.9A} yields for all $n \geq 3$, 
\begin{align}
\begin{split} 
& - \Delta \big[\psi_0 + (h_0 + 0 I_{L^2(\bbR^n)})^{-1} v_1 \phi_0\big] (x)    \no \\
& \quad = - [\Delta \psi_0](x) 
- \Delta_x [r_{0,0} * (v_1 \phi_0)](x)     \no \\
& \quad = - v(x) \psi_0(x) + v_1(x) \phi_0(x)   \no \\
& \quad = - v(x) \psi_0(x) + v(x) \psi_0(x) = 0.
\end{split} 
\end{align}
Consequently,  
\begin{equation}
- \Delta \big[\psi_0 + (h_0 + 0 I_{L^2(\bbR^n)})^{-1} v_1 \phi_0\big] = 0. 
\end{equation}
Since $\psi_0 \in [L^2(\bbR^n)]^N$, and by exactly the same arguments employed in 
\eqref{9.41A}--\eqref{9.42A} also $(h_0 + 0 I_{L^2(\bbR^n)})^{-1} (v_1 \phi_0) \in L^2(\bbR^n)$, one 
concludes that $- \Delta \wti \psi_0 = 0$ in the sense of distributions, where 
\begin{equation}
\wti \psi_0 = \big[\psi_0 + (h_0 + 0 I_{L^2(\bbR^n)})^{-1} v_1 \phi_0\big] \in L^2(\bbR^n).
\end{equation}
Thus, also $\big|\nabla \wti \psi_0 \big| \in L^2(\bbR^n)$ implying $\wti \psi_0 \in W^{2,2}(\bbR^n)$. But then 
$\wti \psi_0 = 0$ since 
$\ker(h_0) = \{0\}$. Hence,
\begin{equation}
\psi_0 = - (h_0 + 0 I_{L^2(\bbR^n)})^{-1} v_1 \phi_0,  
\end{equation} 
implying, $\phi_0 \neq 0$, and 
\begin{align}
0 &= v_2 \wti \psi_0 = v_2 \psi_0 + v_2 (h_0 + 0 I_{L^2(\bbR^n)})^{-1}v_1 \phi_0   \no \\
& = \big[I_{L^2(\bbR^n)} + \ol{v_2 (h_0 + 0 I_{L^2(\bbR^n)})^{-1} v_1}\big] \phi_0, 
\end{align}
that is, 
\begin{equation} 
0 \neq \phi_0 \in 
\ker \big(\big[I_{L^2(\bbR^n)} + \ol{v_2 (h_0 + 0 I_{L^2(\bbR^n)})^{-1} v_1}\big]\big).
\end{equation}  
This concludes the proof. 
\end{proof}

\begin{remark} \lb{r9.8b} 
Employing $\Delta \psi = v \psi$ with $v \in L^r(\bbR^n)$ for $r> n/4$ and $\psi \in L^q(\bbR^n)$ for 
$q \in (3,\infty) \cup \{\infty\}$, respectively, $\Delta \psi = v_1 \phi$, 
with $v_1 \in L^s(\bbR^n)$ for $s > n/2$ and $\phi \in L^2(\bbR^n)$, an application of H\"older's inequality yields 
additional $L^p(\bbR^n)$-properties of $\Delta \psi$, but we omit further details at this point. \hfill $\diamond$
\end{remark}

\begin{remark} \lb{r9.8a} 
$(i)$ For basics on the Birman--Schwinger principle in an abstract context, especially, if the energy parameter 
in the Birman--Schwinger operator belongs to the resolvent set of the unperturbed operator, we refer to \cite{GLMZ05} (cf.\ also \cite{BGHN16}, \cite{GHN15}) and the extensive literature cited therein. 
In the concrete case of Schr\"odinger operators, relations \eqref{9.37}, \eqref{9.38} are discussed at length in \cite{AGH82}, \cite{BGD88}, \cite{BGDW86}, \cite{EGG14}, \cite{EG13}, \cite{ES04}, \cite{ES06},  
\cite{GH87}, \cite{Je80}--\cite{JN01}, \cite{Mu82}, \cite{To17} (see also the list of references quoted therein). \\[1mm] 
$(ii)$ In physical notation (see, e.g., \cite[footnote~3 on p.~300]{Ne02} for details), the zero-energy resonances in Cases $(II)$ and $(IV)$ for $n=3,4$, are $s$-wave resonances (i.e., corresponding to angular momentum zero)  in the case where $V$ is spherically symmetric (see also the discussion in \cite{KS80}). \\[1mm]
$(iii)$ As mentioned in Remark \ref{r5.1}, the absence of zero-energy resonances is well-known in dimensions $n \geq 5$, see \cite{Je80}. \\[1mm] 
$(iv)$ For discussions of the threshold behavior of resolvents of Schr\"odinger operators in dimensions 
$n = 1,2$ (the cases not studied in this paper) we refer to \cite{BGW83}--\cite{BGK87}, \cite{EG13}, \cite{JN01}, \cite{KS80}, \cite{Mu82}, \cite{To17}. The cases 
$n \geq 3$ are treated in \cite{AGH82}, \cite{EGG14}, \cite{EGS09}, \cite{ES04}, \cite{ES06}, \cite{GH87}, \cite{Je80}--\cite{JN01}, \cite{KS80}, \cite{Mu82}, \cite{Ya05a}. 
${}$ \hfill $\diamond$
\end{remark}

\noindent 
{\bf Massless Dirac Operators in $\boldsymbol{\bbR^n}$, $\boldsymbol{n \geq 2}$.}
Next, we turn to the case of massless Dirac operators $H$ in dimension $n \geq 2$ as defined in \eqref{4.2},  
and start by making the following assumptions on the matrix-valued potential $V$. 

\begin{hypothesis} \lb{h9.5}
Let $n \in \bbN$, $n \geq 2$. Assume the a.e.~self-adjoint matrix-valued potential 
$V = \{V_{\ell,m}\}_{1 \leq \ell,m \leq N}$ satisfies for some $C \in (0,\infty)$, 
\begin{align}
\begin{split} 
& V \in [L^{\infty} (\bbR^n)]^{N \times N},       \\
& |V_{\ell,m}(x)| \leq C \langle x \rangle^{- 2} \, 
\text{ for a.e.~$x \in \bbR^n$, $1 \leq \ell,m \leq N$.}    \lb{9.31A}
\end{split} 
\end{align}
In addition, alluding to the polar decomposition of $V(\dott)$ $($i.e., $V(\dott) = U_V(\dott)|V(\dott)|$$)$ in the following symmetrized form $($cf.\ \cite{GMMN09}$)$, we suppose that 
\begin{equation} 
V = V_1^* V_2 = |V|^{1/2} U_V |V|^{1/2}, \text{ where } \, V_1 = V_1^* = |V|^{1/2}, \quad 
V_2 = U_V |V|^{1/2}. 
\end{equation} 
\end{hypothesis} 

We continue with the threshold behavior, that is, the $z=0$ behavior, of $H$: 

\begin{definition} \lb{d9.6} Assume Hypothesis \ref{h9.5}. \\[1mm] 
$(i)$ The point $0$ is called a zero-energy eigenvalue of $H$ if $H \Psi = 0$ has a distributional solution $\Psi$ satisfying 
\begin{equation} 
\Psi \in \dom(H) = [W^{1,2}(\bbR^n)]^N     
\end{equation}
$($equivalently, $\ker(H) \supsetneqq \{0\}$$)$. \\[1mm] 
$(ii)$ The point $0$ is called a zero-energy $($or threshold\,$)$ resonance of $H$ if 
\begin{equation}
\ker\big(\big[I_{[L^2(\bbR^n)]^N} + \ol{V_2 (H_0 - (0 + i 0) I_{[L^2(\bbR^n)]^N})^{-1} V_1^*}\big]\big)  
\supsetneqq \{0\}, 
\end{equation}
and if there exists $0 \neq \Phi \in \ker\big(\big[I_{[L^2(\bbR^n)]^N} + \ol{V_2 (H_0 - (0 + i 0) I_{[L^2(\bbR^n)]^N})^{-1} V_1^*}\big]\big)$ such that $\Psi$ defined by 
\begin{align}
\begin{split} 
\Psi(x) &= - \big((H_0 - (0 + i 0) I_{[L^2(\bbR^n)]^N})^{-1} V_1^* \Phi\big)(x)     \lb{9.46A} \\
& = - i 2^{-1} \pi^{- n/2} \Gamma(n/2) \int_{\bbR^n} d^n y \, |x-y|^{-n} 
[\alpha \cdot (x - y)] V_1(y)^* \Phi(y)    
\end{split} 
\end{align}
$($for a.e.~$x \in \bbR^n$, $n \geq 2$$)$ is a distributional solution of $H \Psi = 0$ satisfying 
\begin{equation} 
\Psi \notin [L^2(\bbR^n)]^N.
\end{equation} 
$(iii)$ $0$ is called a regular point for $H$ if it is neither a zero-energy eigenvalue nor a zero-energy 
resonance of $H$.  
\end{definition}

Additional properties of $\Psi$ are isolated in Theorem \ref{t9.7}.

While the point $0$ being regular for $H$ is the generic situation, zero-energy eigenvalues and/or resonances are exceptional cases. 

Next, we introduce the following convenient abbreviation (for $x, y \in \bbR^n$, $x \neq y$):
\begin{align}
R_{0,0}(x-y) &= \lim_{\substack{z \to 0 \\ z \in \ol{\bbC_+}\backslash\{0\}}} G_0(z;x,y) 
= G_0(0+i0;x,y)   \no \\ 
&= i 2^{-1} \pi^{-n/2} \Gamma(n/2) \, \alpha \cdot \f{(x - y)}{|x - y|^n} \no \\
& = \begin{cases} (2 \pi)^{-1} i \alpha \cdot [\nabla_x \ln(|x-y|)], & n=2, \\
- i \alpha \cdot [\nabla_x g_0(0;x,y)], & n\geq 3.
\end{cases}    \lb{9.46} 
\end{align}

\begin{theorem} \lb{t9.7}
Assume Hypothesis \ref{h9.5}.  \\[1mm] 
$(i)$ If $n=2$, there are precisely four possible cases: \\[1mm]
Case $(I)$: $0$ is regular for $H$. \\[1mm]
Case $(II)$: $0$ is a $($possibly degenerate\footnote{One can show that if $n=2$, the degeneracy in Case $(II)$ is at most two, see \cite{EGG18}.}\,$)$ resonance of $H$. In this case, the resonance 
functions $\Psi$ satisfy  
\begin{align} 
\begin{split} 
& \Psi \in [L^q(\bbR^2)]^2, \quad q \in (2, \infty) \cup \{\infty\},  \quad 
\nabla \Psi \in [L^2(\bbR^2)]^{2 \times 2},    \\ \lb{3.60b}
& \Psi \notin [L^2(\bbR^2)]^2.
\end{split}
\end{align}
Case $(III)$: $0$ is a $($possibly degenerate\,$)$ eigenvalue of $H$. In this case, the corresponding 
eigenfunctions $\Psi \in \dom(H) = \big[W^{1,2}(\bbR^2)\big]^2$ of $H \Psi = 0$ also satisfy 
\begin{equation} 
\Psi \in [L^q(\bbR^2)]^2, \quad q \in [2, \infty) \cup \{\infty\}.    \lb{3.60c}
\end{equation} 
Case $(IV)$: A possible mixture of Cases $(II)$ and $(III)$. \\[1mm]
$(ii)$ If $n \in \bbN$, $n \geq 3$, there are precisely two possible cases: \\[1mm]
Case $(I)$: $0$ is regular for $H$. \\[1mm]
Case $(II)$: $0$ is a $($possibly degenerate\,$)$ eigenvalue of $H$. In this case, the corresponding eigenfunctions $\Psi \in \dom(H) = \big[W^{1,2}(\bbR^n)\big]^N$ of $H \Psi = 0$ also satisfy
\begin{equation}
\Psi \in \big[L^q(\bbR^n)\big]^N, \quad q \in \begin{cases} (3/2, \infty) \cup \{\infty\}, & n=3, \\
(4/3,4), & n=4, \\
(2n/(n+2), 2n/(n-2)), & n \geq 5.
\end{cases}     \lb{3.70A} 
\end{equation}
In particular, there are no zero-energy resonances of $H$ in dimension $n \geq 3$. \\[1mm] 
$(iii)$ The point $0$ is regular for $H$ if and only if 
\begin{equation}
\ker\big(\big[I_{[L^2(\bbR^n)]^N} + \ol{V_2 (H_0 - (0 + i 0) I_{[L^2(\bbR^n)]^N})^{-1} V_1^*}\big]\big) 
= \{0\}. 
\end{equation}
\end{theorem}
\begin{proof}
Since $G_0(0+i0;x,y)$, $x\neq y$, exists for all $n \geq 2$ (cf.\ \eqref{5.18}), the 
Birman--Schwinger eigenvalue equation  
\begin{equation} 
\big[I_{[L^2(\bbR^n)]^N} + \ol{V_2 (H_0 - (0 + i 0) I_{[L^2(\bbR^n)]^N})^{-1} V_1^*}\big] \Phi_0 = 0,  \quad 
0 \neq \Phi_0 \in [L^2(\bbR^n)]^N,   \lb{9.36} 
\end{equation}
gives rise to a distributional zero-energy solution $\Psi_0 \in [L^1_{\loc}(\bbR^n)]^N$ of $H \Psi_0 = 0$ in 
terms of $\Phi_0$ of the form (for a.e.~$x \in \bbR^n$, $n \geq 2$),
\begin{align}
& \Psi_0(x) = - \big((H_0 - (0 + i 0) I_{[L^2(\bbR^n)]^N})^{-1} V_1^* \Phi_0\big)(x)     \no \\
& \hspace*{8.5mm} = - [R_{0,0} * (V_1^* \Phi_0)](x)    \lb{9.36a} \\
& \hspace*{8.5mm} = - i 2^{-1} \pi^{- n/2} \Gamma(n/2) \int_{\bbR^n} d^n y \, |x-y|^{-n} 
[\alpha \cdot (x - y)] V_1(y)^* \Phi_0(y),   \lb{9.37} \\
& \hspace*{8.5mm} = - i 2^{-1} \pi^{- n/2} \Gamma(n/2) \int_{\bbR^n} d^n y \, |x-y|^{-n} 
[\alpha \cdot (x - y)] V_1(y)^* V_2 (y) \Psi_0(y),   \lb{9.37a} \\
& \Phi_0(x) = (V_2 \Psi_0)(x).       \lb{9.38}
\end{align}
In particular, one concludes that $\Psi_0 \neq 0$. Thus, one estimates, 
with $\|V_1(\dott)\|_{\bbC^{N \times N}} \leq c\langle \dott \rangle^{-1}$ and some constant $d_n \in (0, \infty)$, 
\begin{align}
\begin{split}
\|\Psi_0(x)\|_{\bbC^N} &\leq d_n \int_{\bbR^n} d^n y \, |x - y|^{1-n} \langle y \rangle^{-1} 
\|\Phi_0(y)\|_{\bbC^N}   \lb{9.41} \\
& = d_n \cR_{1,n}\big(\langle \dott \rangle^{-1} \|\Phi_0(\dott)\|_{\bbC^N}\big)(x),  \quad x \in \bbR^n.  
\end{split}
\end{align} 
An application of Theorem \ref{t5.6}\,$(ii)$ with $c=0$, $d=1$, $p=p'=2$, and the inequality $1 < n/2$, combined with 
$\|\Phi_0(\, \cdot \,)\|_{\bbC^N} \in L^2(\bbR^n)$, then yield 
\begin{equation}
\|\Psi_0(\, \cdot \,)\|_{\bbC^N} \in L^2(\bbR^n) \, \text{ and hence, } \, \Psi_0 \in [L^2(\bbR^n)]^N, 
\quad n \geq 3.     \lb{9.42} 
\end{equation} 
To prove that actually $\Psi_0 \in \dom(H) = [W^{1,2}(\bbR^n)]^N$, it suffices to argue as follows:  
\begin{equation}
i \alpha \cdot \nabla \Psi_0 = - V \Psi_0 \in [L^2(\bbR^n)]^N    \lb{9.43}
\end{equation}     
in the sense of distributions since $V \in [L^{\infty} (\bbR^n)]^{N \times N}$ and $\Psi_0 \in [L^2(\bbR^n)]^N$.  
Given the fact $\dom(H_0) = \big[W^{1,2}(\bbR^n)\big]^N$ (cf.\ \eqref{2.2}), one concludes  
\begin{equation}
\Psi_0 \in \big[W^{1,2}(\bbR^n)\big]^N, \quad n \geq 3.     \lb{9.45}
\end{equation}

Returning to the estimate \eqref{9.41}, and invoking the Riesz potential $\cR_{1,n}$ (cf.\ Theorem \ref{t5.4A}), 
one obtains (for some constant $\wti d_n \in (0,\infty)$)
\begin{align}
\|\Psi_0(x)\|_{\bbC^N} & \leq d_n \int_{\bbR^n} d^n y \, |x - y|^{1-n} \langle y \rangle^{-1} \|\Phi_0(y)\|_{\bbC^N}  \no \\
& \leq \wti d_n \cR_{1,n} \big(\langle \dott \rangle^{-1} \|\Phi_0(\dott)\|_{\bbC^N}\big)(x),  \quad x \in \bbR^n, 
\lb{3.80A} 
\end{align}
and hence \eqref{5.51} implies (for some constant $\wti C_{p,q,n} \in (0,\infty)$)  
\begin{align}
\|\Psi_0\|_{[L^q(\bbR^n)]^N} & \leq \wti d_n 
\big\|\cR_{1,n} \big(\langle \dott \rangle^{-1} \|\Phi_0(\dott)\|_{\bbC^N}\big)
\big\|_{L^q(\bbR^n)}    \no \\
& \leq \wti C_{p,q,n} \big\|\langle \dott \rangle^{-1} \|\Phi_0(\dott)\|_{\bbC^N}\big\|_{L^p(\bbR^n)}   \no \\
& \leq \wti C_{p,q,n} \big\|\langle \dott \rangle^{-1} \big\|_{L^s(\bbR^n)} \|\|\Phi_0(\dott)\|_{\bbC^N}\|_{L^2(\bbR^n)} 
\no \\
& = \wti C_{p,q,n} \big\|\langle \dott \rangle^{-1}\big\|_{L^s(\bbR^n)} \|\Phi_0\|_{[L^2(\bbR^n)]^N},   \lb{9.61} \\
& \hspace*{-2.75cm} 1 < p < q < \infty, \; p^{-1} = q^{-1} + n^{-1}, \; s = 2qn [2n + 2q - qn]^{-1} \geq 1.  \no
\end{align}
In particular, 
\begin{equation}
p = qn/(n + q), \quad 2n + 2q - qn > 0. 
\end{equation}

\noindent 
$(a)$ The case $n=2$: Then one can choose $q \in (2,\infty)$, hence  
$p = 2q/(q+2) \in (1,2)$, and $s = q > 2$. Thus, \eqref{9.61} and 
$\big\|\langle \dott \rangle^{-1}\big\|_{L^s(\bbR^2)} < \infty$ imply
\begin{equation}
\Psi_0 \in [L^q(\bbR^2)]^N, \quad q \in (2,\infty). 
\end{equation}

Recalling $R_{0,0}(\dott)$ in \eqref{9.46}, this implies 
\begin{align}
\begin{split} 
- i \alpha \cdot \nabla_x R_{0,0}(x-y) = - \Delta_x g_{0}(0;x,y) I_N = \delta(x-y) I_N,&   \lb{9.73}\\
x, y \in \bbR^n, \; x \neq y, \; n \geq 2,&
\end{split} 
\end{align}
in the sense of distributions. Here we abused notation a bit  and denoted also in the case $n=2$, 
\begin{equation}
g_0(0;x,y) = - (2 \pi)^{-1} \ln(|x-y|), \quad x, y \in \bbR^2, \; x \neq y, \; n=2.    \lb{9.74} 
\end{equation}
Thus, one obtains 
\begin{align}
& i \alpha \cdot (\nabla \Psi)(x) = 
- i \alpha \cdot \nabla_x [R_{0,0} * (V_1^* \Phi_0)](x) = - i \alpha \cdot \nabla_x [- i (\alpha \cdot \nabla_x g_0 * (V_1^* \Phi_0)](x)   \no \\
& \quad = (- \Delta_x g_{0} I_N) * (V_1^* \Phi_0)](x) = (V_1^* \Phi_0)(x)  \in [L^2(\bbR^2)]^2,   \lb{9.75} 
\end{align}
proving $\nabla \Psi_0 \in [L^2(\bbR^2)]^{2 \times 2}$, upon employing the fact that $[\alpha \cdot p]^2 = I_N |p|^2$, 
$p \in \bbR^n$.    

To prove that $\Psi\in [L^{\infty}(\bbR^2)]^2$ in \eqref{3.60b} and \eqref{3.60c}, one applies \eqref{9.38} to the inequality in \eqref{9.41}, and then employs the condition $\|V_2(\dott)\|_{\bbC^{2\times 2}}\leq C\langle \dott\rangle^{-1}$ for some constant $C\in (0,\infty)$ to obtain
\begin{equation}
\|\Psi_0(x)\|_{\bbC^2} \leq \widetilde{d}_2 \int_{\bbR^2}d^2y\, |x-y|^{-1}\langle y\rangle^{-2}\|\Psi_0(y)\|_{\bbC^2},\quad x\in \bbR^2,
\end{equation}
where $\widetilde{d}_2\in (0,\infty)$ is an appropriate $x$-independent constant.  By H\"older's inequality,
\begin{equation}\lb{3.81b}
\|\Psi_0(x)\|_{\bbC^2} \leq \widetilde{d}_2\bigg(\int_{\bbR^2}d^2y\, |x-y|^{-3/2}\langle y\rangle^{-3}\bigg)^{2/3} \bigg(\int_{\bbR^2}d^2y\, \|\Psi_0(y)\|_{\bbC^2}^3 \bigg)^{1/3},\quad x\in \bbR^2.
\end{equation}
The second integral on the right-hand side in \eqref{3.81b} is finite since $\Psi_0 \in [L^3(\bbR^2)]^2$.  Choosing $x_1=x$, $k=3/2$, $\ell=0$, $\beta=2$, and $\varepsilon = 1$ in Lemma \ref{l3.12}, one infers that
\begin{equation}\lb{3.82b}
\int_{\bbR^2}d^2y\, |x-y|^{-3/2}\langle y\rangle^{-3} \leq C_{2,3/2,0,2,1},\quad x\in \bbR^2. 
\end{equation}
Hence, the containment $\Psi_0\in [L^{\infty}(\bbR^2)]^2$ follows from \eqref{3.81b} and \eqref{3.82b}.

\noindent 
$(b)$ The case $n \geq 3$: By \eqref{9.45} we know that $\Psi_0 \in \big[W^{1,2}(\bbR^n)\big]^N$. Employing 
the fact that $\Phi_0 = V_2 \Psi_0$ in the first line of \eqref{3.80A}, one obtains 
\begin{align}
\|\Psi_0(x)\|_{\bbC^N} & \leq \wti D_n \int_{\bbR^n} d^n y \, |x - y|^{1-n} \langle y \rangle^{-2} \|\Psi_0(y)\|_{\bbC^N}  \no \\
& =D_n \cR_{1,n} \big(\langle \dott \rangle^{-2} \|\Psi_0(\dott)\|_{\bbC^N}\big)(x),  \quad x \in \bbR^n, 
\lb{9.77A} 
\end{align}
for some constants $\wti D_n, D_n \in (0,\infty)$. Thus, as in \eqref{9.61}, \eqref{5.51} implies for $n \geq 3$, 
\begin{align}
\|\Psi_0\|_{[L^q(\bbR^n)]^N} & \leq D_n 
\big\|\cR_{1,n} \big(\langle \dott \rangle^{-2} \|\Psi_0(\dott)\|_{\bbC^N}\big)
\big\|_{L^q(\bbR^n)}    \no \\
& \leq \wti D_{p,q,n} \big\|\langle \dott \rangle^{-2} \|\Psi_0(\dott)\|_{\bbC^N}\big\|_{L^p(\bbR^n)}   \no \\
& \leq \wti D_{p,q,n} \big\|\langle \dott \rangle^{-2} \big\|_{L^s(\bbR^n)} \|\|\Psi_0(\dott)\|_{\bbC^N}\|_{L^2(\bbR^n)} 
\no \\
& = \wti D_{p,q,n} \big\|\langle \dott \rangle^{-2}\big\|_{L^s(\bbR^n)} \|\Psi_0\|_{[L^2(\bbR^n)]^N},   \lb{9.78A} \\
& \hspace*{-2.75cm} 1 < p < q < \infty, \; p^{-1} = q^{-1} + n^{-1}, \; s = 2qn [2n + 2q - qn]^{-1} \geq 1,   \no
\end{align}
for some constant $\wti D_{p,q,n} \in (0,\infty)$. In particular, one again has $p = qn/(n + q)$ and 
$2n + 2q - qn > 0$. The latter condition implies $q < 2n/(n-2)$. The requirement $p > 1$ results in $q > n/(n-1)$, and the condition $s \geq 1$ yields $q \geq 2n/(3n-2)$ which, however, is superseded by $q > p > 1$. 
Moreover, the requirement $\big\|\langle \dott \rangle^{-2}\big\|_{L^s(\bbR^n)} < \infty$ yields $q > 2n/(n+2)$. 
Putting it all together implies \eqref{3.70A}.

To prove the containment $\Psi\in [L^{\infty}(\bbR^3)]^4$ in \eqref{3.70A}, one applies \eqref{9.38} to the inequality in \eqref{9.41} to obtain
\begin{equation}
\|\Psi_0(x)\|_{\bbC^4} \leq d_3 \int_{\bbR^3}d^3y\, |x-y|^{-2}\langle y\rangle^{-2}\|\Psi_0(y)\|_{\bbC^4},\quad x\in \bbR^3.
\end{equation}
By H\"older's inequality (with conjugate exponents $q'=27/20$ and $q=27/7$), one infers that
\begin{align}
& \|\Psi_0(x)\|_{\bbC^4} \leq d_3\bigg(\int_{\bbR^3}d^3y\, |x-y|^{-27/10}\langle y\rangle^{-27/10} 
\bigg)^{20/27} 
\bigg(\int_{\bbR^3}d^3y\, \|\Psi_0(y)\|_{\bbC^4}^{2/77} \bigg)^{7/27},\no \\
& \hspace*{9.7cm}  x\in \bbR^3.     \lb{3.86b}
\end{align}
The second integral on the right-hand side in \eqref{3.86b} is finite since $\Psi_0\in [L^{27/7}(\bbR^3)]^4$, and the first integral on the right-hand side in \eqref{3.86b} may be estimated by taking $x_1=x$, $k=27/10$, $\ell=0$, $\beta=2$, and $\varepsilon = 7/10$ in Lemma \ref{l3.12}, 
\begin{equation}\lb{3.87b}
\int_{\bbR^3}d^3y\, |x-y|^{-\frac{27}{10}}\langle y\rangle^{-\frac{27}{10}}\leq C_{3,\frac{27}{10},0,2,7/10}, 
\quad x\in \bbR^3\backslash\{0\}.
\end{equation}
Hence, the containment $\Psi_0\in [L^{\infty}(\bbR^3)]^4$ follows from \eqref{3.86b} and \eqref{3.87b}.
We will illustrate in Remark \ref{rA.2} that the same line of reasoning fails for $n \geq 4$. 

Finally (following the proof of \cite[Lemma~7.4]{EGG18}), we show that if $\ker(H) \supsetneqq \{0\}$ then also 
$\ker \big(\big[I_{[L^2(\bbR^n)]^N} + \ol{V_2 (H_0 - (0 + i0) I_{[L^2(\bbR^n)]^N})^{-1} V_1^*}\big]\big) 
\supsetneqq \{0\}$. Indeed, if $0 \neq \Psi_0 \in \ker(H)$, then 
$\Phi_0 := V_2 \Psi_0 = U_V V_1 \Psi_0 \in [L^2(\bbR^n)]^N$ and hence 
$V_1^* \Phi_0 \in [L^2(\bbR^n)]^N$. Then, $H \Psi_0 = 0$ yields 
$i \alpha \cdot \nabla  \Psi_0 = V \Psi_0 = V_1^* V_2 \Psi_0 = V_1^* \Phi_0$. 

Thus, applying \eqref{9.46}, \eqref{9.73}--\eqref{9.74} once again, one obtains for all $n \geq 2$, 
\begin{align}
& - i \alpha \cdot \nabla \big[\Psi_0 + (H_0 - (0 + i 0) I_{[L^2(\bbR^n)]^N})^{-1} V_1^* \Phi_0\big] (x)    \no \\
& \quad = - i [\alpha \cdot \nabla \Psi_0](x) 
- i \alpha \cdot \nabla_x [R_{0,0} * (V_1^* \Phi_0)](x)     \no \\
& \quad = - i [\alpha \cdot \nabla \Psi_0](x) 
- i \alpha \cdot \nabla_x [- i (\alpha \cdot \nabla_x g_0 * (V_1^* \Phi_0)](x)   \no \\
& \quad = - i [\alpha \cdot \nabla \Psi_0](x) + (- \Delta_x g_{0} I_N) * (V_1^* \Phi_0)](x)    \no \\
& \quad = - i [\alpha \cdot \nabla \Psi_0](x) + (V_1^* \Phi_0)(x)    \no \\
& \quad = - V(x) \Psi_0(x) + V(x) \Psi_0(x) = 0.
\end{align}
Consequently,  
\begin{equation}
- i \alpha \cdot \nabla [\Psi_0 + (H_0 - (0 + i 0) I_{[L^2(\bbR^n)]^N})^{-1} V_1^* \Phi_0] = 0,
\end{equation}
implying 
\begin{equation}
\Psi_0 + (H_0 - (0 + i 0) I_{[L^2(\bbR^n)]^N})^{-1} V_1^* \Phi_0 = c^\top, 
\end{equation}
for some $c \in \bbC^N$. Since $\Psi_0 \in [L^2(\bbR^n)]^N$, and by exactly the same arguments employed in 
\eqref{9.41}--\eqref{9.42}, also $R_{0,0} * (V_1^* \Phi_0) \in [L^2(\bbR^n)]^N$, one concludes that $c = 0$ and hence
\begin{equation}
\Psi_0 = - (H_0 - (0 + i 0) I_{[L^2(\bbR^n)]^N})^{-1} V_1^* \Phi_0. 
\end{equation} 
Thus, $\Phi_0 \neq 0$, and 
\begin{align}
0 &= V_2 \Psi_0 + V_2 (H_0 - (0 + i 0) I_{[L^2(\bbR^n)]^N})^{-1}V_1^* \Phi_0   \no \\
&= \big[I_{[L^2(\bbR^n)]^N} + \ol{V_2 (H_0 - (0 + i0) I_{[L^2(\bbR^n)]^N})^{-1} V_1^*}\big] \Phi_0, 
\end{align}
that is, 
\begin{equation} 
0 \neq \Phi_0 \in 
\ker \big(\big[I_{[L^2(\bbR^n)]^N} + \ol{V_2 (H_0 - (0 + i0) I_{[L^2(\bbR^n)]^N})^{-1} V_1^*}\big]\big).
\end{equation}  
This concludes the proof. 
\end{proof}

\begin{remark} \lb{r9.8} 
$(i)$ For basics on the Birman--Schwinger principle in the concrete case of massless Dirac operators see \cite{EGG17}, \cite{EGG18}. \\[1mm] 
$(ii)$ In physical notation, the zero-energy resonances in Cases $(II)$ and $(IV)$ for $n=2$ correspond to eigenvalues $\pm 1/2$ of the spin-orbit operator (cf.\ the operator $S$ in \cite{KOY15}, \cite{KY01}) when 
 $V$ is spherically symmetric, see the discussion in \cite{EGG18}. \\[1mm]
$(iii)$ As mentioned in Remark \ref{r5.1}, the absence of zero-energy resonances is well-known in the 
three-dimensional case $n=3$, see \cite{Ai16}, \cite[Sect.~4.4]{BE11}, \cite{BES08}, \cite{BGW95}, \cite{SU08}, 
\cite{SU08a}, \cite{ZG13}. In fact, for $n=3$ the absence of zero-energy resonances has been shown 
under the weaker decay $|V_{j,k}| \leq C\langle x\rangle^{-1 - \varepsilon}$, $x \in \bbR^3$, in \cite{Ai16}. 
The absence of zero-energy resonances for massless Dirac operators in dimensions $n \geq 4$ as contained in 
Theorem \ref{t9.7}\,$(ii)$ appears to have gone unnoticed in the literature.
${}$ \hfill $\diamond$
\end{remark}
\noindent 
{\bf Massive Dirac Operators in $\boldsymbol{\bbR^n}$, $\boldsymbol{n \geq 3}$.}  
Finally, we turn to the case of massive Dirac operators $H(m)$, $m > 0$, in dimension $n \geq 3$ as 
defined in \eqref{4.8}, and start by making the following assumptions on the matrix-valued potential $V$.

\begin{hypothesis} \lb{Zh9.5}
Let $n \in \bbN$, $n \geq 3$. Assume the a.e.~self-adjoint matrix-valued potential 
$V = \{V_{j,k}\}_{1 \leq j,k \leq N}$ satisfies for some $C \in (0,\infty)$, 
\begin{align}
\begin{split} 
& V \in [L^{\infty} (\bbR^n)]^{N \times N},       \\
& |V_{j,k}(x)| \leq C \langle x \rangle^{- 4} \, 
\text{ for a.e.~$x \in \bbR^n$, $1 \leq j,k \leq N$.}    \lb{Z9.31A}
\end{split} 
\end{align}
In addition, alluding to the polar decomposition of $V(\dott)$ $($i.e., $V(\dott) = U_V(\dott)|V(\dott)|$$)$ in the following symmetrized form $($cf.\ \cite{GMMN09}$)$, we suppose that 
\begin{equation} 
V = V_1^* V_2 = |V|^{1/2} U_V |V|^{1/2}, \text{ where } \, V_1 = V_1^* = |V|^{1/2}, \quad 
V_2 = U_V |V|^{1/2}. 
\end{equation} 
\end{hypothesis} 

We continue with the threshold behavior, that is, the $z=\pm m$ behavior, of $H(m)$: 

\begin{definition} \lb{Zd9.6} Assume Hypothesis \ref{h9.5}. \\[1mm] 
$(i)$ The point $\pm m$ is called a threshold eigenvalue of $H(m)$ if $H(m) \Psi = \pm m \Psi$ has a distributional solution $\Psi$ satisfying 
\begin{equation} 
\Psi \in \dom(H(m)) = [W^{1,2}(\bbR^n)]^N     
\end{equation}
$($equivalently, $\ker(H\mp m I_{[L^2(\bbR^n)]^N}) \supsetneqq \{0\}$$)$. \\[1mm] 
$(ii)$ The point $\pm m$ is called a threshold resonance of $H(m)$ if 
\begin{equation}
\ker\big(\big[I_{[L^2(\bbR^n)]^N} + \ol{V_2 (H_0(m) - (\pm m + i 0) I_{[L^2(\bbR^n)]^N})^{-1} V_1^*}\big]\big) 
\supsetneqq \{0\} 
\end{equation}
and if there exists
\begin{equation}
0 \neq \Phi_{\pm m} \in \ker\big(\big[I_{[L^2(\bbR^n)]^N} + \ol{V_2 (H_0(m) - (\pm m + i 0) I_{[L^2(\bbR^n)]^N})^{-1} V_1^*}\big]\big)
\end{equation}
such that $\Psi_{\pm m}$ defined by 
\begin{align}
\Psi_{\pm m}(x) &= - \big((H_0(m) - (\pm m + i 0) I_{[L^2(\bbR^n)]^N})^{-1} V_1^* \Phi\big)(x)     \no \\
&= -\int_{\bbR^n} d^ny \bigg[4^{-1}\pi^{-n/2}\Gamma((n-2)/2)|x-y|^{2-n}(m\beta \pm mI_{N}) \lb{Z9.46A}\\
&\hspace*{2.2cm}+ i2^{-1}\pi^{-n/2}\Gamma(n/2)\alpha\cdot \frac{(x-y)}{|x-y|^n} \bigg]V_1(y)^* \Phi_{\pm m}(y)\no
\end{align} 
$($for a.e.~$x \in \bbR^n$, $n \geq 2$$)$ is a distributional solution of $H(m) \Psi_{\pm m} = \pm m \Psi_{\pm m}$ satisfying 
\begin{equation} 
\Psi_{\pm m} \notin [L^2(\bbR^n)]^N.
\end{equation} 
$(iii)$ $\pm m$ is called a regular point for $H(m)$ if it is neither a threshold eigenvalue nor a threshold 
resonance of $H(m)$.
\end{definition}

Additional properties of $\Psi_{\pm m}$ are isolated in Theorem \ref{t9.7}.

While the points $\pm m$ being regular for $H(m)$ is the generic situation, threshold eigenvalues and/or resonances at energies $\pm m$ are exceptional cases. 

Next, we introduce the following convenient abbreviation (for $x, y \in \bbR^n$, $x \neq y$):
\begin{align}
R_{0,\pm m}(x-y) &= \lim_{\substack{z \to \pm m \\ z \in \ol{\bbC_+}\backslash\{\pm m\}}} G_0(m,z;x,y) 
= G_0(m,\pm m+i0;x,y)   \no \\ 
&= 4^{-1}\pi^{-n/2}\Gamma((n-2)/2)|x-y|^{2-n}(m\beta \pm mI_{N})\no\\
&\quad + i2^{-1}\pi^{-n/2}\Gamma(n/2)\alpha\cdot \frac{(x-y)}{|x-y|^n}\no\\
&=g_0(0;x,y)(m\beta \pm mI_N) + G_0(0+i0;x,y)\no\\
&=r_{0,0}(x-y)(m\beta \pm mI_N) + R_{0,0}(x-y),\quad n\in \bbN,\, n\geq 3.\lb{Z9.46}
\end{align}

\begin{theorem} \lb{Zt4.3}
Assume Hypothesis \ref{Zh9.5}.  \\[1mm] 
$(i)$ If $n=3,4$, there are precisely four possible cases for each of $\pm m$: \\[1mm]
Case $(I)$: $\pm m$ is regular for $H(m)$. \\[1mm]
Case $(II)$: $\pm m$ is a $($possibly degenerate$)$ resonance of $H(m)$. In this case the resonance 
functions $\Psi_{\pm m}$ satisfy 
\begin{align}
\begin{split} 
& \Psi_{\pm m} \in [L^q(\bbR^n)]^N,\quad q\in \begin{cases} (3,\infty) \cup\{\infty\},& n=3,\\
(2,4),&n=4,
\end{cases}    \\
&\nabla \Psi_{\pm m} \in [L^2(\bbR^n)]^{N \times n}, \quad \Psi_{\pm m} \notin [L^2(\bbR^n)]^N.    \lb{Zz4.11}
\end{split} 
\end{align}
Case $(III)$: $\pm m$ is a $($possibly degenerate\,$)$ eigenvalue of $H(m)$. In this case the corresponding 
eigenfunctions $\Psi_{\pm m} \in \dom(H(m)) = \big[W^{1,2}(\bbR^n)\big]^N$ of 
$H(m) \Psi_{\pm m} = \pm m \Psi_{\pm m}$ also satisfy  
\begin{equation} 
\Psi_{\pm m} \in [L^q(\bbR^n)]^N, \quad q \in \begin{cases} [2,\infty) \cup\{\infty\}, & n=3, \\
[2,4), & n=4. \end{cases}   \lb{Z3.109}
\end{equation} 
Case $(IV)$: A possible mixture of Cases $(II)$ and $(III)$. \\[1mm]
$(ii)$ If $n \in \bbN$, $n \geq 5$, there are precisely two possible cases: \\[1mm]
Case $(I)$: $\pm m$ is regular for $H(m)$. \\[1mm]
Case $(II)$: $\pm m$ is a $($possibly degenerate\,$)$ eigenvalue of $H(m)$. In this case, the corresponding eigenfunctions $\Psi_{\pm m} \in \dom(H(m)) = \big[W^{1,2}(\bbR^n)\big]^N$ of 
$H \Psi_{\pm m} = \pm m \Psi_{\pm m}$ also satisfy
\begin{equation}
\Psi_{\pm m} \in \big[L^q(\bbR^n)\big]^N, \quad q\in
\begin{cases}
(n/(n-1),2n/(n-2)),& 5\leq n\leq 8,\\
(2n/(n+6),2n/(n-4)),& n\geq 9.
\end{cases}    \lb{Z3.70A} 
\end{equation}
In particular, there are no resonances at energies $\pm m$ of $H(m)$ in dimension $n \geq 5$. \\[1mm] 
$(iii)$ The point $\pm m$ is regular for $H(m)$ if and only if 
\begin{equation}
\ker\big(\big[I_{[L^2(\bbR^n)]^N} + \ol{V_2 (H_0(m) - (\pm m + i 0) I_{[L^2(\bbR^n)]^N})^{-1} V_1^*}\big]\big) 
= \{0\}. 
\end{equation}
\end{theorem}
\begin{proof}
Since $G_0(m,\pm m+i0;x,y)$, $x\neq y$, exists for all $n \geq 3$ (cf.\ \eqref{Z9.46}), the 
Birman--Schwinger eigenvalue equation  
\begin{align} 
0&=\big[I_{[L^2(\bbR^n)]^N} + \ol{V_2 (H_0(m) - (\pm m + i 0) I_{[L^2(\bbR^n)]^N})^{-1} V_1^*}\big] \Phi_{\pm m},\lb{Z9.36}\\ 
0 &\neq \Phi_{\pm m} \in [L^2(\bbR^n)]^N,\no
\end{align}
gives rise to a distributional solution $\Psi_{\pm m} \in [L^1_{\loc}(\bbR^n)]^N$ of the equation
\begin{equation}
H(m) \Psi_{\pm m} = \pm m\Psi_{\pm m}
\end{equation}
in terms of $\Phi_{\pm m}$ of the form (for a.e.~$x \in \bbR^n$, $n \geq 3$),
\begin{align}
\Psi_{\pm m}(x) &= - \big((H_0(m) - (\pm m + i 0) I_{[L^2(\bbR^n)]^N})^{-1} V_1^* \Phi_0\big)(x)     \no \\
&= - [R_{0,\pm m} * (V_1^* \Phi_{\pm m})](x)    \lb{Z9.36a} \\
&= - \int_{\bbR^n} d^ny \bigg[4^{-1}\pi^{-n/2}\Gamma((n-2)/2)|x-y|^{2-n}(m\beta \pm mI_{N})\\
&\hspace*{2.2cm}+ i2^{-1}\pi^{-n/2}\Gamma(n/2)\alpha\cdot \frac{(x-y)}{|x-y|^n} \bigg]V_1(y)^* \Phi(y),\lb{Z9.37} \\
\Phi_{\pm m}(x) &= (V_2 \Psi_{\pm m})(x).       \lb{Z9.38}
\end{align}
In particular, one concludes that $\Psi_{\pm m} \neq 0$. Using the inequality $\|V_1(\dott)\|_{\bbC^{N \times N}} \leq c\langle \dott \rangle^{-2}$, one obtains the following estimate for some $d_n,\widetilde{d}_n \in (0, \infty)$ (for a.e.~$x\in \bbR^n$, $n\geq 3$):
\begin{align}
&\|\Psi_{\pm m}(x)\|_{\bbC^N} 
\leq d_n \bigg\{\int_{\bbR^n} d^n y \, |x - y|^{2-n} \langle y \rangle^{-2} 
\|\Phi_{\pm m}(y)\|_{\bbC^N}       \lb{Zz4.19} \\
&\hspace*{2.35cm} + \int_{\bbR^n} d^n y \, |x - y|^{1-n} \langle y \rangle^{-2} 
\|\Phi_{\pm m}(y)\|_{\bbC^N}\bigg\}\no\\
&\quad \leq \widetilde{d}_n\big\{ \cR_{1,n}\big(\langle \,\cdot\, \rangle^{-2}\|\Phi_{\pm m}(\,\cdot\,)\|_{\bbC^N}\big)(x) + \cR_{2,n}\big(\langle \,\cdot\, \rangle^{-2}\|\Phi_{\pm m}(\,\cdot\,)\|_{\bbC^N}\big)(x)\big\}.  \lb{Z9.41}
\end{align} 
The first term on the right-hand side in \eqref{Z9.41} may be estimated as follows:
\begin{align}
&\cR_{1,n}\big(\langle \,\cdot\, \rangle^{-2}\|\Phi_{\pm m}(\,\cdot\,)\|_{\bbC^N}\big)(x) 
=\gamma(1,n)^{-1}\int_{\bbR^n} d^n y \, |x - y|^{1-n} \langle y \rangle^{-2} \|\Phi_{\pm m}(y)\|_{\bbC^N}\no\\
&\quad \leq \gamma(1,n)^{-1} \int_{\bbR^n} d^n y \, |x - y|^{1-n} \langle y \rangle^{-1} \|\Phi_{\pm m}(y)\|_{\bbC^N},   
\quad x\in \bbR^n.    \lb{Z4.22}
\end{align}
An application of Theorem \ref{t5.6}\,$(ii)$ with $c=0$, $d=1$, $p=p'=2$, and the inequality $1 < n/2$, combined with 
$\|\Phi_{\pm m}(\, \cdot \,)\|_{\bbC^N} \in L^2(\bbR^n)$, then yield
\begin{align}
\cR_{1,n}\big(\langle \,\cdot\, \rangle^{-2}\|\Phi_{\pm m}(\,\cdot\,)\|_{\bbC^N}\big)\in L^2(\bbR^n),\quad n\geq 3.\lb{Z4.23}
\end{align}
Similarly, the second term on the right-hand side in \eqref{Z9.41} may be estimated as follows:
\begin{align}
& \cR_{2,n}\big(\langle \,\cdot\, \rangle^{-2}\|\Phi_{\pm m}(\,\cdot\,)\|_{\bbC^N}\big)(x) 
 =\gamma(2,n)^{-1}\int_{\bbR^n} d^n y \, |x - y|^{2-n} \langle y \rangle^{-2} \|\Phi_{\pm m}(y)\|_{\bbC^N},   \no \\
& \hspace*{9.5cm} x\in \bbR^n.     \lb{Z4.20}
\end{align}
An application of Theorem \ref{t5.6}\,$(ii)$ with $c=0$, $d=2$, $p=p'=2$, and the inequality $2 < n/2$, combined with 
$\|\Phi_{\pm m}(\, \cdot \,)\|_{\bbC^N} \in L^2(\bbR^n)$, then yield
\begin{align}
\cR_{2,n}\big(\langle \,\cdot\, \rangle^{-2}\|\Phi_{\pm m}(\,\cdot\,)\|_{\bbC^N}\big) \in L^2(\bbR^n),\quad n\geq 5.\lb{Z4.21}
\end{align}
Thus, \eqref{Z9.41}, \eqref{Z4.21}, and \eqref{Z4.23} imply
\begin{equation}
\|\Psi_{\pm m}(\, \cdot \,)\|_{\bbC^N} \in L^2(\bbR^n) \, \text{ and hence, } \, \Psi_{\pm m} \in [L^2(\bbR^n)]^N, 
\quad n \geq 5.     \lb{Z9.42} 
\end{equation} 
To prove that actually $\Psi_{\pm m} \in \dom(H(m)) = [W^{1,2}(\bbR^n)]^N$ for $n\geq 5$, it suffices to argue as follows:  
\begin{equation}
i \alpha \cdot \nabla \Psi_{\pm m} = - V \Psi_{\pm m}\pm m\Psi_{\pm m} \in [L^2(\bbR^n)]^N    \lb{Z9.43}
\end{equation}     
in the sense of distributions since $V \in [L^{\infty} (\bbR^n)]^{N \times N}$ and $\Psi_{\pm m} \in [L^2(\bbR^n)]^N$.
Given the fact $\dom(H_0(m)) = \big[W^{1,2}(\bbR^n)\big]^N$ (cf.\ \eqref{4.8}), one concludes that 
\begin{equation}
\Psi_{\pm m} \in \big[W^{1,2}(\bbR^n)\big]^N, \quad n \geq 5.     \lb{Z9.45}
\end{equation}

Returning to \eqref{Z9.41}, one applies \eqref{5.51} to estimate (for some constant $C_{p,q,n}^{(1)}\in (0,\infty)$)
\begin{align}
\big\|\cR_{1,n} \big(\langle \dott \rangle^{-2} \|\Phi_{\pm m}(\dott)\|_{\bbC^N}\big)
\big\|_{L^q(\bbR^n)} & \leq C_{p,q,n}^{(1)} \big\|\langle \dott \rangle^{-2} \|\Phi_{\pm m}(\dott)\|_{\bbC^N}\big\|_{L^p(\bbR^n)}   \no \\
& \leq C_{p,q,n}^{(1)} \big\|\langle \dott \rangle^{-2} \big\|_{L^s(\bbR^n)} \big\|\|\Phi_{\pm m}(\dott)\|_{\bbC^N}\big\|_{L^2(\bbR^n)} 
\no \\
& = C_{p,q,n}^{(1)} \big\|\langle \dott \rangle^{-2}\big\|_{L^s(\bbR^n)} \|\Phi_{\pm m}\|_{[L^2(\bbR^n)]^N},   \lb{Z9.61a} \\
& \hspace*{-3.85cm} 1 < p < q < \infty, \; p^{-1} = q^{-1} + n^{-1}, \; s = 2qn [2n + 2q - qn]^{-1} \geq 1,   \no 
\end{align}
so that, in particular, 
\begin{equation}\lb{Z4.30a}
p = qn/(n + q), \quad 2n + 2q - qn > 0. 
\end{equation}
Similarly, one applies \eqref{5.51} to estimate (for some constant $C_{p,q,n}^{(2)}\in (0,\infty)$)
\begin{align}
\big\|\cR_{2,n} \big(\langle \dott \rangle^{-2} \|\Phi_{\pm m}(\dott)\|_{\bbC^N}\big)
\big\|_{L^q(\bbR^n)} & \leq C_{p,q,n}^{(2)} \big\|\langle \dott \rangle^{-2} \|\Phi_{\pm m}(\dott)\|_{\bbC^N}\big\|_{L^p(\bbR^n)}   \no \\
& \leq C_{p,q,n}^{(2)} \big\|\langle \dott \rangle^{-2} \big\|_{L^s(\bbR^n)} \big\|\|\Phi_{\pm m}(\dott)\|_{\bbC^N}\big\|_{L^2(\bbR^n)} 
\no \\
& = C_{p,q,n}^{(2)} \big\|\langle \dott \rangle^{-2}\big\|_{L^s(\bbR^n)} \|\Phi_{\pm m}\|_{[L^2(\bbR^n)]^N},   \lb{Z9.61} \\
& \hspace*{-4.05cm} 1 < p < q < \infty, \; p^{-1} = q^{-1} + 2n^{-1}, \; s = 2qn [2n + 4q - qn]^{-1} \geq 1, \no
\end{align}
so that, in particular, 
\begin{equation}\lb{Z4.32a}
p = qn/(n + 2q), \quad 2n + 4q - qn > 0. 
\end{equation}

We note that if $\Psi_{\pm m}\in [L^2(\bbR^n)]^N$, then an application of \eqref{Z9.38} after the second inequality in \eqref{Z9.61a} yields (for some constant $\widetilde{C}_{p,q,n}^{(1)}\in (0,\infty)$)
\begin{align}
& \big\|\cR_{1,n} \big(\langle \dott \rangle^{-2} \|\Phi_{\pm m}(\dott)\|_{\bbC^N}\big)
\big\|_{L^q(\bbR^n)}     \no \\
& \quad \leq C_{p,q,n}^{(1)} \big\|\langle \dott \rangle^{-2} \|\Phi_{\pm m}(\dott)\|_{\bbC^N}\big\|_{L^p(\bbR^n)}   
\no \\
& \quad \leq C_{p,q,n}^{(1)} \big\| \langle \dott\rangle^{-2}\|V_2(\dott)\|_{\bbC^{N\times N}}\|\Psi_{\pm m}(\dott)\|_{\bbC^N}\big\|_{L^p(\bbR^n)}\no\\
&\quad \leq \widetilde{C}_{p,q,n}^{(1)}\big\|\langle\dott\rangle^{-4}\|\Psi_{\pm m}(\dott)_{\bbC^N}\big\|_{L^p(\bbR^n)}\no\\
&\quad \leq \widetilde{C}_{p,q,n}^{(1)}\big\|\langle \dott\rangle^{-4}\big\|_{L^s(\bbR^n)}\|\Psi_{\pm m}\|_{[L^2(\bbR^n)]^N}, \lb{Z4.49}
\end{align}
under the same conditions in \eqref{Z9.61a} and \eqref{Z9.61a}.  Note that $\|V_2(\dott)\|_{\bbC^{N\times N}}\leq c \langle\dott\rangle^{-2}$ has been used in \eqref{Z4.49}.  Similarly, one obtains (for some constant $\widetilde{C}_{p,q,n}^{(2)}\in (0,\infty)$)
\begin{align}
\big\|\cR_{2,n} \big(\langle \dott \rangle^{-2} \|\Phi_{\pm m}(\dott)\|_{\bbC^N}\big)
\big\|_{L^q(\bbR^n)}\leq \widetilde{C}_{p,q,n}^{(2)}\big\|\langle \dott\rangle^{-4}\big\|_{L^s(\bbR^n)}\|\Psi_{\pm m}\|_{[L^2(\bbR^n)]^N}\lb{Zz4.52}
\end{align}
under the same conditions in \eqref{Z9.61} and \eqref{Z4.32a}.

\noindent 
$(a)$ The case $n=3$:  The first condition in \eqref{Z4.30a} implies $p=3q/(3+q)$.  The condition $p>1$ requires $q\in (3/2,\infty)$, and the condition $s=6q/(6-q)\geq 1$ further requires $q\in (3/2,6)$.  For $q\in (3/2,6)$, one infers $p\in (1,2)$ and $s\in (2,\infty)$; the latter ensures $\|\langle \dott\rangle^{-2}\|_{L^s(\bbR^3)}<\infty$.  Hence, one may take
\begin{equation}
q\in (3/2,6),\quad p=3q/(3+q),\quad s=6q/(6-q)
\end{equation}
throughout \eqref{Z9.61a} to obtain
\begin{equation}\lb{Z4.35}
\cR_{1,3} \big(\langle \dott \rangle^{-2} \|\Phi_{\pm m}(\dott)\|_{\bbC^4}\big) \in L^q(\bbR^3),\quad q\in (3/2,6).
\end{equation}
The first condition in \eqref{Z4.32a} implies $p=3q/(3+2q)$.  The condition $p>1$ requires $q\in (3,\infty)$, which implies $s=6q/(6+q)\in (2,6)$, ensuring $\|\langle \dott\rangle^{-2}\|_{L^s(\bbR^3)}<\infty$.  Hence, one may take
\begin{equation}
q\in (3,\infty),\quad p=3q/(3+2q),\quad s=6q/(6+q)
\end{equation}
throughout \eqref{Z9.61} to obtain
\begin{equation}\lb{Z4.36}
\cR_{2,3} \big(\langle \dott \rangle^{-2} \|\Phi_{\pm m}(\dott)\|_{\bbC^4}\big) \in L^q(\bbR^3),\quad q\in (3,\infty).
\end{equation}
The containments in \eqref{Z4.35} and \eqref{Z4.36}, combined with the estimate in \eqref{Z9.41} imply
\begin{equation}\lb{Z4.37}
\|\Psi_{\pm m}(\, \cdot \,)\|_{\bbC^4} \in L^q(\bbR^3) \, \text{ and hence, } \, \Psi_{\pm m} \in [L^q(\bbR^3)]^4, 
\quad q\in (3,6).
\end{equation} 

To prove that $\nabla \Psi_{\pm m}\in [L^2(\bbR^3)]^{4 \times 3}$, it suffices to argue as follows.  By \eqref{Z9.36a}, \eqref{Zz3.31}, and \eqref{9.75}, one computes
\begin{align}
\begin{split} 
-i\alpha \cdot [\nabla \Psi_{\pm m}](x) 
&= i\alpha \cdot 4^{-1}\pi^{-n/2}\Gamma((n-2)/2)\int_{\bbR^3}d^3y\, |x-y|^{1-n}\frac{(x-y)}{|x-y|}   \\
&\quad \times V_1^*(y)\Phi_{\pm m}(y) [m\beta \pm mI_4]+ (V_1^*\Phi_{\pm m})(x),\quad x\in \bbR^3.\lb{Zz4.43}
\end{split} 
\end{align}
In addition,
\begin{align}
&\bigg\| \int_{\bbR^3}d^3y\, |x-y|^{1-n}\frac{(x-y)}{|x-y|}V_1^*(y)\Phi_{\pm m}(y)\bigg\|_{\bbC^4}\no\\
&\quad \leq \int_{\bbR^3}d^3y\, |x-y|^{1-n}\|V_1^*(y)\|_{\bbC^{4 \times 4}}\|\Phi_{\pm m}(y)\|_{\bbC^4}\no\\
&\quad \leq c\int_{\bbR^3}d^3y\, |x-y|^{1-n}\langle y \rangle^{-2}\|\Phi_{\pm m}(y)\|_{\bbC^4},
\end{align}
and an application of Theorem \ref{t5.6}\,$(ii)$ with $c=0$, $d=1$, $p=p'=2$ then implies
\begin{equation}\lb{Zz4.45}
\int_{\bbR^3}d^3y\, |\dott-y|^{1-n}\langle y \rangle^{-2}\|\Phi_{\pm m}(y)\|_{\bbC^4}\in [L^2(\bbR^3)]^4.
\end{equation}
Since $V_1^*\Phi_{\pm m}\in [L^3(\bbR^3)]^{4}$, the containment in \eqref{Zz4.45} and the identity in \eqref{Zz4.43} 
imply $\nabla \Psi_{\pm m}\in [L^2(\bbR^3)]^{4 \times 3}$.  

Finally, if $\Psi_{\pm m}\in [L^2(\bbR^3)]^4$, then together with \eqref{Z4.37}, one obtains \eqref{Z3.109} for $n=3$.

If $\Psi_{\pm m}$ is a resonance function (resp., eigenfunction), then it has been shown that $\Psi_{\pm m}\in [L^q(\bbR^3)]^4$ for all $q\in (3,6)$ (resp., $q\in [2,6)$).  To prove that $\Psi_{\pm m}\in [L^{\infty}(\bbR^3)]^4$, one applies \eqref{Z9.38} and the condition $\|V_2(\dott)\|_{\bbC^{4\times 4}}\leq C\langle\dott\rangle^{-2}$ for some $C\in (0,\infty)$ in \eqref{Zz4.19} to obtain
\begin{align}
\|\Psi_{\pm m}(x)\|_{\bbC^4}&\leq \widetilde{d}_3\bigg\{ \int_{\bbR^3}d^3y\, |x-y|^{-1}\langle y\rangle^{-4}\|\Psi_{\pm m}(y)\|_{\bbC^4}\no\\
&\quad + \int_{\bbR^3}d^3y\, |x-y|^{-2}\langle y\rangle^{-4}\|\Psi_{\pm m}(x)\|_{\bbC^4}\bigg\},\quad x\in \bbR^3.   \lb{3.142b}
\end{align}
Here $\widetilde{d}_3\in (0,\infty)$ is an appropriate $x$-independent constant.  The first integral on the right-hand side in \eqref{3.142b} may be estimated with H\"older's inequality (with conjugate exponents $p=5/4$ and $p'=5$) as follows:
\begin{align}
&\int_{\bbR^3}d^3y\, |x-y|^{-1}\langle y\rangle^{-4}\|\Psi_{\pm m}(y)\|_{\bbC^4}\lb{3.143b}\\
&\quad \leq \bigg(\int_{\bbR^3}d^3y\, |x-y|^{-5/4}\langle y\rangle^{-5}\bigg)^{4/5}\bigg(\int_{\bbR^3}d^3y\, \|\Psi_{\pm m}(y)\|_{\bbC^4}^5\bigg)^{1/5},\quad x\in \bbR^3.\no
\end{align}
The second integral on the right-hand side in \eqref{3.143b} is finite since a resonance function (resp., eigenfunction) satisfies $\Psi_{\pm m}\in [L^5(\bbR^3)]^4$.  To estimate the first integral on the right-hand side in \eqref{3.143b}, ones applies Lemma \ref{l3.12} with the choices $x_1=x$, $k=5/4$, $\ell=0$, $\beta=4$, and 
$\varepsilon =1$, so that
\begin{equation}\lb{3.144b}
\int_{\bbR^3}d^3y\, |x-y|^{-5/4}\langle y\rangle^{-5} \leq C_{3,5/4,0,4,1},\quad x\in \bbR^3.
\end{equation}
Along similar lines, the second integral on the right-hand side in \eqref{3.142b} may be estimated with H\"older's inequality (with conjugate exponents $p=5/4$ and $p'=5$) as follows:
\begin{align}
&\int_{\bbR^3}d^3y\, |x-y|^{-2}\langle y\rangle^{-4}\|\Psi_{\pm m}(y)\|_{\bbC^4}\lb{3.145b}\\
&\quad \leq \bigg(\int_{\bbR^3}d^3y\, |x-y|^{-5/2}\langle y\rangle^{-5}\bigg)^{4/5}\bigg(\int_{\bbR^3}d^3y\, \|\Psi_{\pm m}(y)\|_{\bbC^4}^5\bigg)^{1/5},\quad x\in \bbR^3.\no
\end{align}
An application of Lemma \ref{l3.12} with $x_1=x$, $k=5/2$, $\ell=0$, $\beta=4$, and $\varepsilon = 1$, yields
\begin{equation}\lb{3.146b}
\int_{\bbR^3}d^3y\, |x-y|^{-5/2}\langle y\rangle^{-5} \leq C_{3,5/2,0,4,1},\quad x\in \bbR^3.
\end{equation}
Therefore, combining \eqref{3.142b}, \eqref{3.143b}, \eqref{3.144b}, \eqref{3.145b}, and \eqref{3.146b}, one obtains
\begin{align}
\|\Psi_{\pm m}(x)\|_{\bbC^4}\leq \widetilde{d}_3\Big\{C_{3,5/4,0,4,1}^{4/5} + C_{3,5/2,0,4,1}^{4/5} \Big\} \|\Psi_{\pm m}\|_{[L^5(\bbR^3)]^4},\quad x\in \bbR^3,
\end{align}
and it follows that $\Psi_{\pm m}\in [L^{\infty}(\bbR^3)]^4$, settling the case $n=3$. We will illustrate in 
Remark \ref{rA.3} that the same line of reasoning fails for $n \geq 4$. 

\noindent 
$(b)$ The case $n=4$:  The first condition in \eqref{Z4.30a} implies $p=4q/(4+q)$.  The condition $p>1$ requires $q\in(4/3,\infty)$, and the condition $s=4q/(4-q)\geq 1$ further requires $q\in (4/3,4)$.  For $q\in (4/3,4)$, one infers $p\in (1,2)$ and $s\in (2,\infty)$; the latter ensures $\|\langle \dott \rangle^{-2}\|_{L^s(\bbR^4)}<\infty$.  Hence, one may take
\begin{equation}
q\in (4/3,4),\quad p=4q/(4+q),\quad s=4q/(4-q)
\end{equation}
throughout \eqref{Z9.61a} to obtain
\begin{equation}\lb{Z4.43}
\cR_{1,4} \big(\langle \dott \rangle^{-2} \|\Phi_{\pm m}(\dott)\|_{\bbC^{4}}\big) \in L^q(\bbR^4),\quad q\in (4/3,4).
\end{equation}

The first condition in \eqref{Z4.32a} implies $p=2q/(2+q)$.  The condition $p>1$ requires $q\in (2,\infty)$, and the condition on $s$ in \eqref{Z9.61} reduces to $s=q$, so the requirement $s\geq 1$ is satisfied for all $q\in (2,\infty)$.  For $q\in (2,\infty)$, $p\in (1,2)$, and $s\in (2,\infty)$; the latter ensures $\|\langle \dott\rangle^{-2}\|_{L^s(\bbR^4)}<\infty$.  Hence, one may take
\begin{equation}
q\in (2,\infty),\quad p=2q/(2+q),\quad s=q
\end{equation}
throughout \eqref{Z9.61} to obtain
\begin{equation}\lb{Z4.45}
\cR_{2,4} \big(\langle \dott \rangle^{-2} \|\Phi_{\pm m}(\dott)\|_{\bbC^{4}}\big) \in L^q(\bbR^4),\quad q\in (2,\infty).
\end{equation}

The containments in \eqref{Z4.43} and \eqref{Z4.45}, combined with the estimate in \eqref{Z9.41} imply
\begin{equation}
\|\Psi_{\pm m}(\dott)\|_{\bbC^{4}} \in L^q(\bbR^4) \, \text{ and hence, } \, \Psi_{\pm m} \in [L^q(\bbR^4)]^{4}, 
\quad q\in (2,4).    \lb{Z3.148}
\end{equation}

One proves $\nabla \Psi_{\pm m}\in [L^2(\bbR^4)]^{4 \times 4}$ in a manner entirely analogous to \eqref{Zz4.43}--\eqref{Zz4.45}.  Finally, if $\Psi_{\pm m}\in [L^2(\bbR^4)]^{4}$, then together with \eqref{Z3.148}, one obtains \eqref{Z3.109} for $n=4$.  Therefore, the $n=4$ case is settled.

\noindent 
$(c)$ The case $n\geq 5$:  The first condition in \eqref{Z4.30a} implies $p=qn/(n+q)$.  The condition $p>1$ requires $q\in (n/(n-1),\infty)$.  The condition $2n+2q-qn>0$ further requires $q<2n/(n-2)$,  and the condition $s=2qn/(2n+2q-qn)\geq 1$ further requires $q\geq 2n/(3n-2)$.  Moreover, $\|\langle \dott\rangle^{-4}\|_{L^s(\bbR^n)}<\infty$ if and only if $n/4<s=2qn/(2n+2q-qn)$, which requires $q>2n/(n+6)$.  Thus, one may take
\begin{align}
\begin{split}
&q\in \begin{cases}
(n/(n-1),2n/(n-2)),& 5\leq n\leq 8,\\
(2n/(n+6),2n/(n-2)),&n\geq 9,
\end{cases}\\
&p = qn/(n+q),\quad s=2qn/(2n+2q-qn)
\end{split}\lb{Zz4.53}
\end{align}
throughout \eqref{Z4.49} to obtain
\begin{equation}\lb{Zz4.48}
\cR_{1,n} \big(\langle \dott \rangle^{-2} \|\Phi_{\pm m}(\dott)\|_{\bbC^{N}}\big) \in L^q(\bbR^n),\quad q\in \begin{cases}
(n/(n-1),2n/(n-2)),& 5\leq n\leq 8,\\
(2n/(n+6),2n/(n-2)),&n\geq 9.
\end{cases}
\end{equation}

The first condition in \eqref{Z4.32a} implies $p=qn/(n+2q)$.  The condition $p>1$ requires $q\in (2/(n-2),\infty)$.  The condition $2n+4q-qn>0$ further requires $q<2n/(n-4)$, and the condition $s=2qn/(2n+4q-qn)\geq 1$ further requires $q\geq 2n/(3n-4)$.  Moreover, $\|\langle \dott\rangle^{-4}\|_{L^2(\bbR^n)}<\infty$ if and only if $n/4<s=2qn/(2n+4q-qn)$, which requires $q>2n/(n+4)$.  Hence, one may take
\begin{equation}
q\in (2n/(n+4),2n/(n-4)),\quad p=qn/(n+2q),\quad s=2qn/(2n+4q-qn)
\end{equation}
throughout \eqref{Z9.61} to obtain
\begin{equation}\lb{Z4.50}
\cR_{2,n} \big(\langle \dott \rangle^{-2} \|\Phi_{\pm m}(\dott)\|_{\bbC^{N}}\big) \in L^q(\bbR^n),\quad q\in (2n/(n+8),2n/(n-4)).
\end{equation}

The containments in \eqref{Zz4.48} and \eqref{Z4.50}, combined with the estimate in \eqref{Z9.41}, imply
\begin{align}
\begin{split}
\|\Psi_{\pm m}(\dott)\|_{\bbC^N} \in L^q(\bbR^n) \, \text{ and hence, } \, \Psi_{\pm m} \in [L^q(\bbR^n)]^N,&  \\
q\in \begin{cases}
(n/(n-1),2n/(n-2)),& 5\leq n\leq 8,\\
(2n/(n+6),2n/(n-4)),& n\geq 9.
\end{cases}&  
\end{split}
\end{align}

Finally (following the proof of \cite[Lemma~7.4]{EGG18}), we show that if
\begin{equation}
\ker(H(m)\mp mI_{[L^2(\bbR^n)]^N}) \supsetneqq \{0\},
\end{equation}
then also 
\begin{equation}
\ker \big(\big[I_{[L^2(\bbR^n)]^N} + \ol{V_2 (H_0(m) - (\pm m + i0) I_{[L^2(\bbR^n)]^N})^{-1} V_1^*}\big]\big) 
\supsetneqq \{0\}.
\end{equation}
Indeed, if $0 \neq \Psi_{\pm m} \in \ker(H(m)\mp mI_{[L^2(\bbR^n)]^N})$, then 
\begin{equation}
\Phi_{\pm m} := V_2 \Psi_{\pm m} = U_V V_1 \Psi_{\pm m} \in [L^2(\bbR^n)]^N,
\end{equation}
and hence $V_1^* \Phi_{\pm m} \in [L^2(\bbR^n)]^N$. Therefore, $H(m) \Psi_{\pm m} = \pm m \Psi_{\pm m}$ yields 
\begin{align}
i \alpha \cdot \nabla  \Psi_{\pm m} &= (m\beta \mp mI_N)\Psi_{\pm m} + V \Psi_{\pm m}\no\\
&= (m\beta \mp mI_N)\Psi_{\pm m} + V_1^* V_2 \Psi_{\pm m}\no\\
&= (m\beta \mp mI_N)\Psi_{\pm m} + V_1^*\Phi_{\pm m}.
\end{align}

Hence,
\begin{align}
&(-i\alpha\cdot \nabla + m\beta \mp mI_N)\big[\Psi_{\pm m} + (H_0(m) - (\mp m + i 0) I_{[L^2(\bbR^n)]^N})^{-1} V_1^* \Phi_{\pm m}\big] (x)\no\\
&\quad= [(-i\alpha\cdot \nabla  +m\beta \mp mI_N)\Psi_{\pm m}](x)\no\\
&\qquad  + (-i\alpha\cdot \nabla  +m\beta \mp mI_N)[R_{0,\pm m}\ast (V_1^*\Phi_{\pm m})](x)\no\\
&\quad = -(V_1^*\Phi_{\pm m})(x) + [((-i\alpha\cdot \nabla  +m\beta \mp mI_N)R_{0,\pm m})\ast (V_1^*\Phi_{\pm m})](x)\no\\
&\quad = -(V_1^*\Phi_{\pm m})(x) + (V_1^*\Phi_{\pm m})(x) = 0,\quad x\in \bbR^n.
\end{align}
In addition, $\Psi_{\pm m}\in [L^2(\bbR^n)]^N$, and the same arguments as those in \eqref{Z4.22}--\eqref{Z4.23} and \eqref{Zz4.52}--\eqref{Zz4.53} (which now extend to include $n\in \{3,4\}$, due to the assumption that $\Psi_{\pm m}\in [L^2(\bbR^n)]^N$) imply  $[R_{0,\pm m}\ast (V_1^*\Phi_{\pm m})]\in [L^2(\bbR^n)]^N$, so that
\begin{equation}
\big[\Psi_{\pm m} + (H_0(m) - (\mp m + i 0) I_{[L^2(\bbR^n)]^N})^{-1} V_1^* \Phi_{\pm m}\big] \in [L^2(\bbR^n)]^N.
\end{equation}
It follows that
\begin{equation}
\wti \Psi_{\pm m} := \big[\Psi_{\pm m} + (H_0(m) - (\mp m + i 0) I_{[L^2(\bbR^n)]^N})^{-1} V_1^* \Phi_{\pm m}\big]=0;
\end{equation}
otherwise, $\pm m$ is an eigenvalue of $H_0(m)$ and $\wti \Psi_{\pm m}$ is a corresponding eigenfunction.  However, this contradicts the fact that the spectrum of $H_0(m)$ is purely absolutely continuous.  Hence,
\begin{equation}
\Psi_{\pm m} = - (H_0(m) - (\mp m + i 0) I_{[L^2(\bbR^n)]^N})^{-1} V_1^* \Phi_{\pm m}.
\end{equation} 
Thus, $\Phi_{\pm m} \neq 0$, and 
\begin{align}
0 &= V_2 \Psi_{\pm m} + V_2 (H_0(m) - (\pm m + i 0) I_{[L^2(\bbR^n)]^N})^{-1}V_1^* \Phi_{\pm m}   \no \\
&= \big[I_{[L^2(\bbR^n)]^N} + \ol{V_2 (H_0(m) - (\pm m + i0) I_{[L^2(\bbR^n)]^N})^{-1} V_1^*}\big] \Phi_{\pm m}, 
\end{align}
that is, 
\begin{equation} 
0 \neq \Phi_{\pm m} \in 
\ker \big(\big[I_{[L^2(\bbR^n)]^N} + \ol{V_2 (H_0(m) - (\pm m + i0) I_{[L^2(\bbR^n)]^N})^{-1} V_1^*}\big]\big).
\end{equation}  
This concludes the proof. 
\end{proof}

\begin{remark} \lb{r9.8c} 
$(i)$ For basics on the Birman--Schwinger principle in the concrete case of massive Dirac operators,    see \cite{EGG17}, \cite{EG17}--\cite{EGT19}, \cite{Kl85}. \\[1mm] 
$(ii)$ In physical notation (see, e.g.,  \cite{KOY15}, \cite{KY01}, \cite{Kl90}, \cite[Sect.~4.6]{Th92} for details), the threshold resonances at energies $\pm m$ in Cases $(II)$ and $(IV)$ for $n=3$ correspond to eigenvalues $\pm 1$ of the spin-orbit operator (cf, \cite{KOY15}, \cite{KY01}, \cite[eq.~(4.105), p.~125]{Th92}) in the case where $V$ is spherically symmetric, see the discussion in \cite{EGT19}.  \\[1mm] 
$(iii)$ The absence of threshold resonances for massive Dirac operators in dimensions $n \geq 5$ as contained in Theorem \ref{Zt4.3}\,$(ii)$ appears to have gone unnoticed in the literature.
${}$ \hfill $\diamond$
\end{remark}

\appendix
\section{Some Remarks on $L^{\infty}(\bbR^n)$-Properties \\ of Threshold Eigenfunctions} \lb{sA}
\renewcommand{\theequation}{A.\arabic{equation}}
\renewcommand{\thetheorem}{A.\arabic{theorem}}
\setcounter{theorem}{0} \setcounter{equation}{0}

In this appendix we collect some negative results on $L^{\infty}(\bbR^n)$-properties of threshold 
eigenfunctions. We hope to return to this issue at a later date. 

We start with the case of Schr\"odinger operators:

\begin{remark} \lb{rA.1} 
We briefly indicate why the line of reasoning used in the proof of $\psi \in L^{\infty}(\bbR^n)$ for $3 \leq n \leq 7$ 
in Theorem \ref{t9.7a} is bound to fail for $n \geq 8$. 

Suppose $n\geq 8$.  Then \eqref{3.54c} continues to hold and by \eqref{3.49A}, $\psi\in L^q(\bbR^n)$ for all
\begin{equation}\lb{3.61cc}
q\in (2n/(n+4),2n/(n-4)).
\end{equation}
One notes that $q$ satisfies \eqref{3.61cc} if and only if its conjugate exponent $q':= q/(q-1)$ satisfies
\begin{equation}\lb{3.62c}
q' \in (2n/(n+4),2n/(n-4)).
\end{equation}
Applying H\"older's inequality with conjugate exponents $q'$ and $q$ on the right-hand side in \eqref{3.54c} yields
\begin{align}
|\psi(x)|\leq \widetilde{d}_n \bigg(\int_{\bbR^n}d^ny\, |x-y|^{(2-n)q'}\langle y\rangle^{-4q'} \bigg)^{1/q'}\bigg(\int_{\bbR^n}d^ny\,|\psi(y)|^q \bigg)^{1/q},\quad x\in \bbR^n.\lb{3.63c}
\end{align}
The second integral on the right-hand side in \eqref{3.63c} is finite for any choice of $q'$ which satisfies \eqref{3.62c}. When applying Lemma \ref{l3.12} to estimate the first integral on the right-hand side in \eqref{3.63c}, one must choose $k=(n-2)q'$.  The hypotheses in Lemma \ref{l3.12} require $k\in [0,n)$.  Hence, one must choose $q'$ so that $(n-2)q' < n$, that is, one must choose $q' < n/(n-2)$.  However, there is no $q'$ which simultaneously satisfies \eqref{3.62c} and $q' < n/(n-2)$ since
\begin{equation}
\frac{2n}{n+4} - \frac{n}{n-2} = \frac{n(n-8)}{(n-2)(n+4)} > 0 \, \text{ implies } \, \frac{n}{n-2}<\frac{2n}{n+4}.
\end{equation}
Thus, $q' < n/(n-2)$ yields $q' < 2n/(n+4)$, so that $q'$ does not satisfy \eqref{3.62c}.   \hfill $\diamond$
\end{remark}

Next, we illustrate the case of massless Dirac operators:

\begin{remark} \lb{rA.2} 
We briefly indicate why the line of reasoning used in the proof of $\Psi \in L^{\infty}(\bbR^n)^N$ for $n = 2,3$ 
in Theorem \ref{t9.7} is bound to fail for $n \geq 4$. 

We start with the case $n=4$. Then the inequality in \eqref{9.41} implies, when combined with \eqref{9.38},
\begin{align}\lb{3.88b}
\|\Psi_0(x)\|_{\bbC^4}\leq \widetilde{d}_4 \int_{\bbR^4}d^4y\, |x-y|^{-3}\langle y\rangle^{-2}\|\Psi_0(y)\|_{\bbC^4},\quad x\in \bbR^4.
\end{align}
In order to invoke the property $\Psi_0\in [L^q(\bbR^4)]^4$, $q\in (4/3,4)$, one applies H\"older's inequality with conjugate indices $q'$ and $q$ on the right-hand side in \eqref{3.88b} to obtain
\begin{equation}\lb{3.89b}
\|\Psi_0(x)\|_{\bbC^4}\leq \widetilde{d}_4 \bigg(\int_{\bbR^4}d^4y\, |x-y|^{-3q'}\langle y\rangle^{-2q'} \bigg)^{1/q'}\bigg(\int_{\bbR^4}d^4y\, \|\Psi_0(y)\|_{\bbC^4}^q \bigg)^{1/q},\quad x\in \bbR^4.
\end{equation}
The second integral on the right-hand side in \eqref{3.89b} is finite for any $q\in (4/3,4)$.  The requirement $q\in (4/3,4)$ implies
\begin{equation}\lb{3.90b}
q'\in (4/3,4).
\end{equation}
In order to apply Lemma \ref{l3.12} to the first integral on the right-hand side in \eqref{3.89b}, one would choose $k=3q'$ and $\ell=0$.  The hypotheses (viz., $k\in [0,n)$) of Lemma \ref{l3.12} require
\begin{equation}
k=3q'<4,
\end{equation}
which is satisfied if and only if $q'<4/3$. However, the condition $q'<4/3$ is incompatible with \eqref{3.90b}.  

An analogous problem is also encountered for all other dimensions $n\geq 5$:  Indeed, for $n\geq 5$, \eqref{9.38} and \eqref{9.41} combine to yield
\begin{equation}\lb{3.92b}
\begin{split}
\|\Psi_0(x)\|_{\bbC^N}\leq \widetilde{d}_n \bigg(\int_{\bbR^n}d^ny\, |x-y|^{(1-n)q'}\langle y\rangle^{-2q'} \bigg)^{1/q'}\bigg(\int_{\bbR^n}d^ny\, \|\Psi_0(y)\|_{\bbC^N}^q \bigg)^{1/q},&\\
x\in \bbR^n,&
\end{split}
\end{equation}
for conjugate exponents $q'$ and $q$.  The second integral on the right-hand side in \eqref{3.92b} is finite for all $q\in (2n/(n+2),2n/(n-2))$.  For $q\in (2n/(n+2),2n/(n-2))$, one infers
\begin{equation}\lb{3.93b}
q'\in (2n/(n+2),2n/(n-2)).
\end{equation}
To apply Lemma \ref{l3.12} to the first integral on the right-hand side in \eqref{3.92b}, one must choose $k=(n-1)q'$.  The condition $k\in [0,n)$ in Lemma \ref{l3.12} requires $(n-1)q'<n$, which is equivalent to
\begin{equation}\lb{3.94b}
q'<\frac{n}{n-1}.
\end{equation}
However, the condition in \eqref{3.94b} is incompatible with \eqref{3.93b} in the sense that there is no $q'\in [1,\infty)$ which satisfies both \eqref{3.93b} and \eqref{3.94b} simultaneously. \hfill $\diamond$
\end{remark} 

Finally, we also illustrate the case of massive Dirac operators:

\begin{remark} \lb{rA.3} 
We briefly indicate why the line of reasoning used in the proof of $\Psi \in L^{\infty}(\bbR^3)^4$  
in Theorem \ref{Zt4.3} is bound to fail for $n \geq 4$. 

For simplicity we focus again on the case $n=4$. 
Applying \eqref{Z9.38} and the condition $\|V_2(\dott)\|_{\bbC^{4\times 4}}\leq C\langle\dott\rangle^{-2}$ for some $C\in (0,\infty)$ in \eqref{Zz4.19}, one obtains
\begin{align}
\|\Psi_{\pm m}(x)\|_{\bbC^{4}}&\leq \widetilde{d}_4\bigg\{ \int_{\bbR^4}d^4y\, |x-y|^{-2}\langle y\rangle^{-4}\|\Psi_{\pm m}(y)\|_{\bbC^{4}}\no\\
&\quad + \int_{\bbR^4}d^4y\, |x-y|^{-3}\langle y\rangle^{-4}\|\Psi_{\pm m}(x)\|_{\bbC^{4}}\bigg\},\quad x\in \bbR^4,\lb{3.148b}
\end{align}
where $\widetilde{d}_4\in (0,\infty)$ is an appropriate $x$-independent constant.

Applying H\"older's inequality (with conjugate exponents $p=3/2$ and $p'=3$) to the first integral on the right-hand side in \eqref{3.148b}, one obtains
\begin{align}
&\int_{\bbR^4}d^4y\, |x-y|^{-2}\langle y\rangle^{-4}\|\Psi_{\pm m}(y)\|_{\bbC^{4}}\no\\
&\quad \leq \bigg(\int_{\bbR^4}d^4y\, |x-y|^{-3}\langle y\rangle^{-6}\bigg)^{2/3}\bigg(\int_{\bbR^4}d^4y\, \|\Psi_{\pm m}(y)\|_{\bbC^{4}}^3\bigg)^{1/3},\quad x\in \bbR^4.\lb{3.149b}
\end{align}
The second integral on the right-hand side in \eqref{3.149b} is finite since a resonance function (resp., eigenfunction) satisfies $\Psi_{\pm m}\in [L^3(\bbR^4)]^{4}$. By Lemma \ref{l3.12} with $x_1=x$, $k=3$, $\ell=0$, and $\beta=5$,
\begin{equation}\lb{3.150b}
\int_{\bbR^4}d^4y\, |x-y|^{-3}\langle y\rangle^{-6} \leq C_{4,3,0,5,1},\quad x\in \bbR^4.
\end{equation}
Applying H\"older's inequality (with conjugate indices $q'$ and $q$) to the second integral on the right-hand side in \eqref{3.148b}, one obtains
\begin{align}
&\int_{\bbR^4}d^4y\, |x-y|^{-3}\langle y\rangle^{-4}\|\Psi_{\pm m}(y)\|_{\bbC^{4}}\lb{3.151b}\\
&\quad \leq \bigg(\int_{\bbR^4}d^4y\, |x-y|^{-3q'}\langle y\rangle^{-4q'}\bigg)^{1/q'}\bigg(\int_{\bbR^4}d^4y\, \|\Psi_{\pm m}(y)\|_{\bbC^{4}}^q\bigg)^{1/q},\quad x\in \bbR^4.\no
\end{align}
The second integral on the right-hand side in \eqref{3.151b} is finite for all $q\in (2,4)$.  The requirement $q\in (2,4)$ implies
\begin{equation} \lb{3.152b}
q'\in (4/3,2).
\end{equation}
In order to apply Lemma \ref{l3.12} to the first integral on the right-hand side in \eqref{3.151b}, one would choose $k=3q'$ and $\ell=0$.  The hypotheses (viz., $k\in [0,n)$) of Lemma \ref{l3.12} require
\begin{equation}
k=3q'<4,
\end{equation}
which is satisfied if and only if $q'<4/3$.  However, the condition $q'<4/3$ is incompatible with \eqref{3.152b}. 
\hfill $\diamond$ 
\end{remark} 

\medskip

\noindent
{\bf Acknowledgments.} 
We are indebted to Alan Carey, Will Green, Jens Kaad, Galina Levitina, Denis Potapov, and Fedor Sukochev 
for many helpful discussions on this subject. We also thank the referee for a critical reading of our manuscript. 

 
 
\end{document}